\documentclass[11pt,thmsa]{article}%
\usepackage{amsmath}
\usepackage{amsfonts}
\usepackage{amssymb}
\usepackage{graphicx}%
\newtheorem{theorem}{Theorem}[section]

\newtheorem{lemma}[theorem]{Lemma}

\newtheorem{question}[theorem]{Question}
\newtheorem{proposition}[theorem]{Proposition}

\newenvironment{proof}[1][Proof]{\noindent\textbf{#1.} }{\ \rule{0.5em}{0.5em}}

\begin{document}
\title{Homogeneous Finsler spaces with only one orbit of prime closed geodesics\thanks{Supported by NSFC (No.  11771331) and Beijing Natural Science Foundation
(No.  1182006).}}
\author{Ming Xu
\\ \\
School of Mathematical Sciences\\
Capital Normal University\\
Beijing 100048, P. R. China\\
Email:mgmgmgxu@163.com
}
\date{}
\maketitle
\begin{abstract}
When a closed Finsler manifold admits continuous isometric actions, estimating the number of orbits of prime closed geodesics seems a more reasonable substitution for estimating the number of prime closed
geodesics. To generalize the works of H. Duan, Y. Long, H.B. Rademacher, W. Wang and others on the existence of two prime closed geodesics to the equivariant situation, we purpose the question if a closed Finsler manifold has only one orbit of prime closed geodesic if and only if it is a compact rank-one Riemannian symmetric space. In this paper, we study this problem in homogeneous Finsler geometry, and get a positive answer
when the dimension is even or the metric is reversible. We guess the rank inequality and algebraic techniques in this paper may continue to play an important role
for discussing our question in the non-homogeneous situation.

{\bf Key words:} Homogeneous Finsler space, closed geodesic, compact rank-one symmetric space, connected isometry group, Killing vector field.

{\bf MSC(2010):} 53C60, 53C30, 53C22.
\end{abstract}

\section{Introduction}

In Riemannian geometry, it has been conjectured for many decades
that on any closed manifold $M$ with $\dim M>1$, there exist
infinitely many prime closed geodesics. In Finsler geometry, this is not true
because of the Katok spheres \cite{Katok-1973} found in 1973. Katok spheres are Randers spheres of constant flag curvature which were much recently classified in \cite{BRS-2004}. Based on the
Katok spheres, D.V. Anozov purposed another conjecture, claiming the existence of $2[\frac{n+1}{2}]$ prime closed geodesics on
the Finsler sphere $(S^n,F)$ \cite{An-1974}. See  \cite{BL-2010} and \cite{Wa2012} for some recent progress on this conjecture.

Generally speaking, finding the first prime closed geodesic
on a compact Finsler manifold is relatively easy (see \cite{Fe1965} or \cite{Kl-1978}).
Finding the second is already a hard problem if no topological
obstacle from \cite{GM-1969} and \cite{PS-1976} is accessible. It
was relatively recent that H. Duan and Y. Long \cite{DL-2007} and H.B. Rademacher \cite{Ra-2010}\cite{Ra-2017} provided different proofs of the following theorem.
\begin{theorem}\label{cite-thm-1}
A bumpy and irreversible Finsler metric on a sphere $S^n$
of dimension $n\geq 3$ carries two prime closed geodesics.
\end{theorem}

More generally, when $S^n$ is changed to other compact manifold,
H. Duan, Y. Long and W. Wang proved the following theorem in
\cite{DLW-2016}.
\begin{theorem}\label{cite-thm-2}
There exist always at least two prime closed geodesics on
every compact simply connected bumpy irreversible Finsler manifold
$(M,F)$.
\end{theorem}

In this paper, we will assume the Finsler manifold $(M,F)$ admits
nontrivial continuous isometries and discuss an equivalent analog of above theorems. Some thought and technique were purposed in \cite{BFIMZ} and further developed in \cite{Xu-2018-1} and \cite{Xu-2018-2}, while studying the geodesics
in a Finsler sphere of constant curvature.

It was suggested in \cite{Xu-2018-2} that,
when the connected isometry group $G=I_0(M,F)$ has a positive dimension, estimating the number of prime closed
geodesic seems more reasonable to be switched to estimating the number of orbits of prime closed geodesics, with respect to the action of $\hat{G}=G\times S^1$ (the precise description for this action will be explained at the end of Section 2).
Though there are examples
of compact Finsler manifolds with only one orbit of prime closed geodesics, they are very rare. The only known examples are compact rank-one Riemannian symmetric spaces, i.e. $S^n$, $\mathbb{C}\mathrm{P}^n$, $\mathbb{H}\mathrm{P}^n$
and $\mathbb{O}\mathrm{P}^2$ when simply connected, and $\mathbb{R}\mathrm{P}^n$ othewise, all endowed with the their standard metrics.

Based on above observations and inspired by Theorem \ref{cite-thm-1}
and Theorem \ref{cite-thm-2}, we would like to ask
\begin{question}\label{main-question}
Assume $(M,F)$ is a closed connected Finsler manifold such that
$G=I_0(M,F)$ has a positive dimension and $(M,F)$ has only
one $\hat{G}$-orbit of prime closed geodesics. Must $M$ be
one of the compact rank-one Riemannian symmetric spaces?
\end{question}

We will show some clue for a positive answer to Question \ref{main-question}.

We assume $(M,F)$ is a Finsler manifold as
described in Question \ref{main-question}, and discuss its properties. In particular, we prove that each closed geodesic on $(M,F)$ is homogeneous, i.e. the orbit of a one-parameter subgroup in $G=I_0(M,F)$ (see Lemma \ref{lemma-2}). It implies that the set of all prime closed geodesics on $(M,F)$ is a $G$-orbit.

The union $N$ of all the closed geodesics in $M$ is crucial for discussing Question \ref{main-question}.
It is a $G$-orbit in $M$, which can be presented as $N=G/H$. Notice that the $G$-action on $N$ is almost effective.
The first important theorem we prove in this paper is the rank inequality for $N$, i.e. $\mathrm{rk}G\leq\mathrm{rk}H+1$. To be
precise, $\mathrm{rk}G=\mathrm{rk}H$ when $\dim N$ is even,
and $\mathrm{rk}G=\mathrm{rk}H+1$ otherwise (see Theorem \ref{thm-rank-inequality}).

It is an important problem to explore if or when $N$ is totally
geodesic in $(M,F)$. We will discuss this problem in subsequent works.
In this paper, we only discuss the case that $N=M$, and present a
classification for the compact connected homogeneous Finsler space $(M,F)=(G/H,F)$ with $G=I_0(M,F)$ and only one orbit of prime closed geodesics.

The two classification theorems, Theorem \ref{classification-thm-1} and Theorem \ref{classification-thm-2}
can be summarized as following.

\begin{theorem}\label{main-thm}
Assume $(M,F)$ is a compact connected homogeneous Finsler space with only one orbit
of prime closed geodesics. If $\dim M$ is odd, we further assume
$F$ is reversible. Then $(M,F)$ must be
a Riemannian symmetric $\mathrm{S}^n$, $\mathbb{R}\mathrm{P}^n$,
$\mathbb{C}\mathrm{P}^n$, $\mathbb{H}\mathrm{P}^n$ or
$\mathbb{O}\mathrm{P}^2$.
\end{theorem}

The proof of Theorem \ref{main-thm} (i.e. the proofs of Theorem \ref{classification-thm-1} and Theorem \ref{classification-thm-2}) mainly uses the rank inequality
in Theorem \ref{thm-rank-inequality} and the discussion for compact Lie algebras. Amazingly, it rhythms with the classification of positively curved homogeneous Finsler spaces
\cite{XD-normal-homogeneous}\cite{XD-towards-odd}\cite{XDHH}\cite{XZ-2017}.
Many algebraic techniques are borrowed from those works.

There might be an alternative approach proving Theorem \ref{main-thm}, by the following theorem of J. McCleary and W. Ziller in \cite{MZ-1987}.

\begin{theorem}\label{cite-thm-3}
Let $M$ be a compact connected simply connected homgeneous space
which is not diffeomorphic to a symmetric space of rank 1. Then
the betti numbers $b_i(\Lambda M,\mathbb{Z}_2)$ are unbounded.
\end{theorem}

In this topological approach, we might need the
non-degenerate condition for all orbits of closed geodesics to
apply a generalized Gromoll-Meyer's theorem; see Theorem 3.1 in
\cite{Ra-1989}. More discussions are needed to cover the non-simply connected case and the gap from
determining the manifold to determining the metric. After all,
Theorem \ref{cite-thm-3} is a very hard topological theorem.
Our proof of Theorem \ref{main-thm}, which mainly uses Lie theory, seems much more fundamental.

Theorem \ref{main-thm} answers Question \ref{main-question} affirmatively in homogeneous Riemannian geometry, as well as for
many homogeneous Finsler spaces. Many techniques for proving
Theorem \ref{main-thm} might be
extended to non-homogeneous context and play an important
role in future works on Question \ref{main-question}.

At the end, we discuss the special case of homogeneous Finsler spheres as the application of these techniques. We prove that a homogeneous Finsler sphere $(M,F)$ has
only one orbit of prime closed geodesics only when it is a
Riemannian symmetric sphere (see Proposition \ref{prop-final}). In this case, we do not need to assume the reversibility of $F$ in advance.

This paper is organized as following. In Section 2, we summarize
some backgroup knowledge in general and homogeneous Finsler geometry. In Section 3, we prove the rank inequality, i.e. Theorem \ref{thm-rank-inequality}. In Section 4,
we discuss some examples of homogeneous Finsler spaces, with only one, or with more orbits of prime closed geodesics. In Section 5,
we introduce the algebraic setup and some orthogonality lemmas in homogeneous Finsler geometry. In Section 6, we classify even dimensional homogeneous Finsler spaces with only one orbit of prime closed geodesics. In Section 7 and Section 8, we classify odd dimensional reversible homogeneous Finsler spaces
with only one orbit of prime closed geodesics.

{\bf Acknowledgement.}
The author would like to sincerely thank Wolfgang Ziller, Hans-Bert Rademacher, Shaoqiang Deng, Huagui Duan and Zaili Yan for the helpful discussions.

\section{Preliminaries}

Here we briefly summarize some fundamental knowledge on general and homogeneous Finsler geometry. See \cite{BCS-2000} and \cite{De2012} for more details.

Unless otherwise specified, we will only discuss
closed connected smooth manifold in this paper.

A {\it Finsler metric} on an $n$-dimensional manifold $M$ is
a continuous function $F:TM\rightarrow[0,+\infty)$, satisfying
the following conditions:
\begin{description}
\item{\rm (1)} $F$ is positive and smooth on the slit tangent bundle $TM\backslash0$;
\item{\rm (2)} $F$ is positively homogeneous of degree one, i.e.
for any $x\in M$, $y\in T_xM$ and $\lambda\geq 0$, we have
$F(x,\lambda y)=\lambda F(x,y)$;
\item{\rm (3)} $F$ is strictly convex, i.e. for any standard
local chart $x=(x^i)\in M$ and $y=y^j\partial_{x^j}\in T_xM$,
the Hessian matrix $(g_{ij}(x,y))=(\frac12[F^2(x,y)]_{y^iy^j})$
is positive definite whenever $y\neq 0$.
\end{description}
We will also call $(M,F)$ a {\it Finsler space} or a
{\it Finsler manifold}. The restriction of $F$ to each
tangent space is called a {\it Minkowski norm}.

A Finsler metric $F$ is called {\it reversible} if
for any $x\in M$ and $y\in T_xM$, $F(x,y)=F(x,-y)$.

The Hessian matrix $(g_{ij}(x,y))$, where $y\in T_xM$ is nonzero, defines an inner product on $T_xM$, i.e. for any $u=u^i\partial_{x^i}$ and $v=v^j\partial_{x^j}$ in $T_xM$,
$$\langle u,v\rangle^F_y=\frac{1}{2}\frac{\partial^2}{\partial s\partial t}[F^2(y+su+tv)]_{y^iy^j}=g_{ij}(x,y)u^iv^j.$$
Sometimes, we simply denote this inner product as $g_y^F$.

A {\it geodesic} $c(t)$ on $(M,F)$ is a nonconstant smooth curve satisfying
the local minimizing principle for the arch length functional.
Usually we parametrize it to have a constant positive speed. Then
the curve $(c(t),\dot{c}(t))$ in $TM\backslash 0$ is an integration curve of the geodesic spray vector field
$\mathbf{G}=y^i\partial_{x^i}-2\mathbf{G}^i\partial_{y^i}$ where
$$\mathbf{G}_i=\frac14g^{il}([F^2]_{x^ky^l}y^k-[F^2]_{x^l}). $$

A geodesic is called {\it reversible} if it is still a geodesic with its direction reversed. A closed geodesic is called {\it prime} if it is not the multiple rotation of another. When we count the closed
geodesics, we only count the prime ones. On the other hand, we specify the directions, i.e. a prime reversible closed geodesic is counted as two. Different closed geodesics are {\it geometrically the same} if their images are the same subsets in $M$.

The isometry group $I(M,F)$ of the compact connected
Finsler space $(M,F)$
is a compact Lie group \cite{DH2002}. Its identity component $I_0(M,F)$
is called the {\it connected isometry group}. We call $(M,F)$ {\it homogeneous} if
$I_0(M,F)$ acts transitively on $M$.
A homogeneous Finsler space $(M,F)$ may have different presentations $M=G/H$ where $G$ is a closed connected
subgroup of $I_0(M,F)$ which actions transitively, and $H$
is the compact isotropy subgroup at $o=eH\in G/H=M$. Denote
$\mathfrak{g}=\mathrm{Lie}(G)$, $\mathfrak{h}=\mathrm{Lie}(H)$,
and $\mathfrak{m}$ any $\mathrm{Ad}(H)$-invariant complement of
$\mathfrak{h}$ in $\mathfrak{g}$, then we call $\mathfrak{g}=\mathfrak{h}+\mathfrak{m}$ a {\it reductive
decomposition} for $G/H$. The $G$-invariant Finsler metric $F$
is one-to-one determined by its restriction to $\mathfrak{m}=T_o(G/H)$, which is an $\mathrm{Ad}(H)$-invariant
Minkowski norm.

The Lie algebra $\mathfrak{g}$ of $G=I_0(M,F)$ can be identified
with the linear space
of all Killing vector field on $(M,F)$. For any vector $u\in\mathfrak{g}$,
we denote $X^u$ the Killing vector field that $u$ defines on $(M,F)$. We denote $\mathrm{rk}G=\mathrm{rk}\mathfrak{g}$ the {\it rank} of the compact Lie group $G$ and its Lie algebra, which is the dimension of the maximal torus or the Cartan subalgebra.

Isometries and Killing vector fields play an important role
in studying geodesics, with the following
two frequently used lemmas.

\begin{lemma}\label{preparation-lemma-totally-geodesic}
In a closed Finsler space $(M,F)$,
the common fixed point set of a family of isometries in $I(M,F)$, or the common zero point set of a family of Killing vector fields of $(M,F)$, is a finite disjoint union of connected imbedded
totally geodesic submanifolds.
\end{lemma}

If a totally geodesic submanifold has a positive dimension, then
its geodesics, with respect to the submanifold metric,
are also geodesics for the ambient Finsler space.

\begin{lemma}\label{preparation-lemma-Killing-geodesic}
Let $X$ be a Killing vector field on the Finsler space $(M,F)$,
and $x\in M$ a critical point for the function $f(\cdot)=F(X(\cdot))$ with $X(x)\neq 0$. Then the integration
curve of $X$ at $x$ is a geodesic. In particular, when $X$
generates an $\mathrm{S}^1$-subgroup in $I_0(M,F)$, its integration curve
at $x$ is a closed geodesic.
\end{lemma}

Lemma \ref{preparation-lemma-totally-geodesic} is a well known easy fact in Riemannian and Finsler geometry, so we skip its proof. See \cite{De2008} for the case of zero point sets of Killing vector fields.
Lemma \ref{preparation-lemma-Killing-geodesic} follows immediately
Lemma 3.1 in \cite{DX-2014}. We call a geodesic {\it homogeneous} (or non-homogeneous)
if and only if it is (or is not, respectively) the orbit of a one-parameter
subgroup in $I_0(M,F)$. Thus Lemma \ref{preparation-lemma-Killing-geodesic} provides homogeneous
geodesics from Killing vector fields which are orbits of $\mathrm{S}^1$-subgroups in $I_0(M,F)$.

Killing vector fields which generate $\mathrm{S}^1$-subgroups in $I_0(M,F)$
can be easily found according to the following lemma.

\begin{lemma}\label{preparation-lemma-dense-rational-lie-vector}
Let $G$ be a compact connected Lie group, and $\mathfrak{g}$ its
Lie algebra. Then the subset $\mathcal{S}$ of all vectors in $\mathfrak{g}$ which generate $\mathrm{S}^1$-subgroups in $G$ is a dense subset in $\mathfrak{g}$.
\end{lemma}

\begin{proof} The subset $\mathcal{S}$ in the lemma is $\mathrm{Ad}(G)$-invariant.
Using the conjugation theorem, we only need to prove
$\mathcal{S}\cap\mathfrak{t}$ is dense in $\mathfrak{t}$,
where $\mathfrak{t}$ is a Cartan subalgbra in $\mathfrak{g}$,
generating a maximal torus. The statement is then obvious.
\end{proof}

In \cite{XZ-2017}, we have applied the following fixed point set
technique based on Lemma \ref{preparation-lemma-totally-geodesic}.
Let $(M,F)=(G/H,F)$ be a homogeneous Finsler space with a compact $G$. For any subset $\mathcal{L}\subset H$, we denote
$\mathrm{Fix}(\mathcal{L},M)$ the common fixed point set of $\mathcal{L}$ on $M$. Obviously $o=eH\in \mathrm{Fix}(\mathcal{L},M)$, so we will denote
$\mathrm{Fix}_o(\mathcal{L},M)$ the connected component of
$\mathrm{Fix}(\mathcal{L},M)$ containing $o$.
By Lemma \ref{preparation-lemma-totally-geodesic},
$\mathrm{Fix}_o(\mathcal{L},M)$ is totally geodesic in $(M,F)$.
Furthermore, $\mathrm{Fix}_o(\mathcal{L},M)=G'\cdot o$, where
$G'$ is the identity component of the centralizer $C_G(\mathcal{L})$ for $\mathcal{L}$ in $G$. Because $F'=F|_{\mathrm{Fix}_o(\mathcal{L},M)}$ is $G'$-invariant, so we may present
$\mathrm{Fix}(\mathcal{L},M)$ as a homogeneous Finsler space
$(G'/G'\cap H,F')$.

At the end of this section. We give the precise description for orbits of prime closed geodesics.

On the free loop space $\Lambda M$ of all the piecewise smooth
curves $c(t):\mathrm{S}^1=\mathbb{R}/\mathbb{Z}\rightarrow M$ on the Finsler space $(M,F)$, we
have the canonical action of $\hat{G}=I_0(M,F)\times \mathrm{S}^1$ defined by
\begin{equation}\label{hat-G-action-on-loops}
((\rho,t')\cdot c)(t)=\rho(c(t+t')) \mbox{ for any }t,t'\in\mathbb{R}/\mathbb{Z}\mbox{ and }\rho\in I_0(M,F).
\end{equation}
Obviously
this action preserves the subset of all the prime closed geodesics. Each $\hat{G}$-orbit of prime closed geodesics is
a finite dimensional submanifold in $\Lambda M$.

\section{Rank inequality}

Let $(M,F)$ be a closed connected Finsler manifold such that $\dim M>1$ and
the connected isometry group $G=I_0(M,F)$ has a positive dimension.
Then the existence of two prime closed geodesics follows immediately.

\begin{lemma}\label{lemma-1}
Assume $(M,F)$ is a closed connected Finsler manifold such that $G=I_0(M,F)$ has a positive dimension. Then there exists two distinct prime closed geodesics.
\end{lemma}

\begin{proof}
Because $G=I_0(M,F)$ is a compact connected Lie group with a positive dimension, we can find a vector $u\in\mathfrak{g}=\mathrm{Lie}(G)$, which generates an $\mathrm{S}^1$-subgroup. Let $x_1\in M$
be a maximum point of the function $f_1(\cdot)=F(X^u(\cdot))$,
and $x_2$ a maximum point of the function $f_2(\cdot)=F(-X^u(\cdot))$. Then the integration curve of $X^u$ at
$x_1$ and the integratin curve of $-X^u$ at $x_2$ are two distinct
closed geodesics of $(M,F)$.
\end{proof}

We further assume

{\bf Assumption (I):} The Finsler space $(M,F)$ has only
one $\hat{G}$-orbit of prime closed geodesics, where $\hat{G}=G\times \mathrm{S}^1$ acts on the closed geodesics as in (\ref{hat-G-action-on-loops}).

When Assumption (I) is satisfied, by
Lemma \ref{lemma-1} and the connectedness of $\hat{G}$,
there must exist infinitely many geometrically distinct closed geodesics, i.e. the
$\hat{G}$-orbit of prime closed
geodesics, as a submanifold in $\Lambda M$, has a dimension bigger than 1.

Assumption (I) implies the following immediate consequences.

\begin{lemma}\label{lemma-no-abel-factor}
Assume $(M,F)$ with $\dim M>1$ and $\dim I_0(M,F)>0$ is a closed connected Finsler space satisfying Assumption (I), then $G=I_0(M,F)$ is semi-simple.
\end{lemma}

\begin{proof}
Assume conversely that $G=I_0(M,F)$ has a center of positive dimension,
which corresponds to the center $\mathfrak{c}(\mathfrak{g})$ of $\mathfrak{g}=\mathrm{Lie}(G)$. We can find a nonzero vector $u\in\mathfrak{c}(\mathfrak{g})$, which generates an $\mathrm{S}^1$-subgroup.
Let $X^u$ be the Killing vector field on $(M,F)$ defined by $u$.
In the proof of Lemma \ref{lemma-1}, we have shown two different prime closed geodesics $\gamma_1$ and $\gamma_2$, which are integration
curves of $X^u$ and $-X^u$ respectively. Because $u$ commutes with
$\mathfrak{g}$, for each $g\in G=I_0(M,F)$, $g_*X=X$. So each prime closed geodesic in the $\hat{G}$-orbit of $\gamma_1$ is also
an integration curve of $X^u$, i.e. $\gamma_2$ belongs to another orbit. This is a contradiction to Assumption (I), which ends the proof of the lemma.
\end{proof}

\begin{lemma}\label{lemma-2}
Assume $(M,F)$ with $\dim M>1$ and $\dim I_0(M,F)>0$ is a closed connected Finsler space satisfying Assumption (I), then
each closed geodesic on $(M,F)$
is homogeneous, i.e. the orbit of a one-parameter subgroup in $G$.
\end{lemma}

\begin{proof}
Assume conversely that there exists a non-homogeneous closed geodesic on $(M,F)$.
Then Assumption (I) implies that all closed geodesics
on $(M,F)$ are non-homogeneous. Because $G=I_0(M,F)$ is a compact connected Lie group with a positive dimension, we can
find a nonzero vector $u\in\mathfrak{g}=\mathrm{Lie}(G)$ which generates an $\mathrm{S}^1$-subgroup.
Denote $X^u$ the Killing vector field defined by $u$. By Lemma \ref{preparation-lemma-Killing-geodesic}, at any maximum $x\in M$ for the function $f(\cdot)=F(X^u(\cdot))$, the integration curve of $X^u$ provides a homogeneous prime closed geodesic. This is a contradiction, which
ends the proof of the Lemma.
\end{proof}

Lemma \ref{lemma-2} implies that the $\mathrm{S}^1$-action along each
prime closed geodesic $c(t)$ shifting the variable $t$ can be
achieved by the $G$-action, so we have immediate that $\hat{G}$-orbit of prime closed geodesics in Assumption (I) is also a $G$-orbit. Further more, the union $N$ of all the closed geodesics
is a $G$-orbit in $M$, and
the induced submanifold metric $F|_N$ is
$G$-invariant.

\begin{lemma} \label{lemma-3}
Assume $(M,F)$ with $\dim M>1$ and $\dim I_0(M,F)>0$ is a closed connected Finsler space
satisfying Assumption (I). Let $N$ be the union of all the closed geodesics.
Then the $G$-action on $N$ is almost effective.
\end{lemma}

\begin{proof}
Assume conversely that the $G$-action on $N$ is not almost effective, i.e. there exists a closed subgroup $G'$ in $G$ with
a positive dimension, which acts trivially on $N$. Then we can find a vector
$u\in\mathfrak{g}'=\mathrm{Lie}(G')$ which generates an $\mathrm{S}^1$-subgroup, and defines a Killing
vector field $X^u$ on $(M,F)$.
By Lemma \ref{preparation-lemma-Killing-geodesic},
the integration curve of $X^u$ at any maximum point $x\in M$
provides a prime closed geodesic of $(M,F)$ outside $N$. This is a contradiction with Assumption (I) which proves the lemma.
\end{proof}

We denote $H$ the isotropy subgroup at some $o\in M$, then we can
identify $N$ with $G/H$ with $o=eH\in G/H$. By Lemma \ref{lemma-3}, $\mathfrak{h}=\mathrm{Lie}(H)$ contains no nonzero ideals of $\mathfrak{g}=\mathrm{Lie}(G)$.
We fix a reductive decomposition $\mathfrak{g}=\mathfrak{h}+\mathfrak{m}$ for $G/H$ which is orthogonal with respect to a chosen bi-invariant inner product on $\mathfrak{g}$. When we identify $\mathfrak{m}$ with the tangent
space $T_oM$, the $\mathrm{Ad}(H)$-action on $\mathfrak{m}$ coincides with the isotropy action.
 The key observation is the following rank inequality.

\begin{theorem}\label{thm-rank-inequality}
Assume $(M,F)$ with $\dim M>1$ is a closed connected Finsler space with a connected isometry group $G=I_0(M,F)$ of positive dimension and only one $\hat{G}$-orbit of prime closed geodesics. Then $N=G/H$, the union of all prime closed geodesics on $(M,F)$, satisfies $\mathrm{rk}\mathfrak{g}\leq\mathrm{rk}\mathfrak{h}+1$. To be more
precise, $\mathrm{rk}\mathfrak{g}=\mathrm{rk}\mathfrak{h}$ when $\dim N$ is even, and $\mathrm{rk}\mathfrak{g}=\mathrm{rk}\mathfrak{h}+1$ when $\dim N$ is odd.
\end{theorem}

\begin{proof} Let $c(t):\mathbb{R}/\mathbb{Z}\rightarrow M$ be any prime closed geodesic passing $o=eH\in M$. Define two subgroups of $G$,
\begin{eqnarray*}
H_1&=& \{\rho\in I_0(M,F)|\exists t_0\in\mathbb{R}/\mathbb{Z}\mbox{ with }
\rho(c(t))=c(t+t_0),\forall t\},\mbox{ and}\\
H_2&=&\{\rho\in I_0(M,F)| \rho(c(t))=c(t),\forall t\}.
\end{eqnarray*}
Obviously, both are compact, and $H_2=H_1\cap H$ is a normal subgroup of $H_1$ with an $\mathrm{S}^1$-quotient. Denote their Lie algebras as $\mathfrak{h}_i=\mathrm{Lie}(H_i)$ for $i=1$ and $2$ respectively.

Firstly, we claim

{\bf Claim 1:}
the union of all the $\mathrm{Ad}(G)$-conjugations
of $\mathfrak{h}_1$, i.e. $\mathcal{U}=\cup_{g\in G}\mathrm{Ad}(g)\mathfrak{h}_1$, is a closed subset in $\mathfrak{g}$.

{\bf Proof of Claim 1.}
Let $\{\mathrm{Ad}(g_i)x_i\}\subset\mathcal{U}$
be a convergent sequence in $\mathfrak{g}$, with
$g_i\in G$ and $x_i\in\mathfrak{h}_1$ for each $i\in\mathbb{N}$.
By the compactness of $G$, we may assume $\lim_{i\rightarrow\infty} g_i=g\in G$ by taking a subsequence. The
$\mathrm{Ad}(G)$-actions preserve the chosen bi-invariant inner product on $\mathfrak{g}$, so $\{x_i\}$ is a bounded sequence in
$\mathfrak{h}_1$. Taking a subsequence, we may further assume
$\lim_{i\rightarrow \infty}x_i=x\in\mathfrak{h}_1$. To summarize,
we have $\lim_{i\rightarrow\infty}\mathrm{Ad}(g_i)x_i=\mathrm{Ad}(g)x\in
\mathcal{U}$, which proves the closeness of $\mathcal{U}$, i.e. Claim 1.

Secondly, we claim

{\bf Claim 2:}
$\mathcal{U}=\cup_{g\in G}\mathrm{Ad}(g)\mathfrak{h}_1$ coincides with $\mathfrak{g}$.

{\bf Proof of Claim 2.}
Assume conversely $\mathcal{U}\neq\mathfrak{g}$. By Lemma \ref{preparation-lemma-dense-rational-lie-vector} and Claim 1,
we can find a vector $u\in\mathfrak{g}\backslash\mathcal{U}$
which generates an $\mathrm{S}^1$-subgroup. By Assumption (I),
the Killing vector field $X^u$ on $(M,F)$ is
not tangent to any prime closed geodesics of $(M,F)$. By
Lemma \ref{preparation-lemma-Killing-geodesic} again, its integration curve at the maximum point of $f(\cdot)=F(X(\cdot))$
provides a prime closed geodesic outside $N$. This is a contradiction which ends the proof of Claim 2.

Finally, we finish the proof of the theorem. By Claim 2,
$\mathfrak{h}_1$ contains a generic vector in $\mathfrak{g}$
which generates a dense one-parameter subgroup in a maximal torus
of $G$. So $H_1$ must contain a maximal torus of $G$, i.e.
$\mathrm{rk}\mathfrak{h}_1=\mathrm{rk}\mathfrak{g}$. Then we have
$$\mathrm{rk}\mathfrak{h}\geq\mathrm{rk}\mathfrak{h}_2=
\mathrm{rk}\mathfrak{h}_1-1=\mathrm{rk}\mathfrak{g}-1,$$
which proves the first statement in the theorem.

The other statements follow the first one, and the same argument in the classification of positively curved homogeneous spaces. See
for example the proof of Theorem 4.1 in \cite{Wa1972}.

This ends the proof of the theorem.
\end{proof}

\label{section-rank-inequality}
\section{Homogeneous examples for Assumption (I)}
\label{Section-homogeneous-examples}

Since this section, we concentrate on compact connected homogeneous Finsler
spaces satisfying Assumption (I), i.e. with only one orbit
of prime closed geodesics.

Before the systematical discussion, we study some important examples, satisfying or not satisfying Assumption (I).

{\bf Example 1.} All compact simply connected Riemannian rank-one symmetric spaces satisfy Assumption (I). They can be listed as
\begin{eqnarray}\label{list-cross}
\mathrm{S}^n=\mathrm{SO}(n+1)/\mathrm{SO}(n)\mbox{ with } n>1, & &
\mathbb{C}\mathrm{P}^n={SU}(n+1)/S(\mathrm{U}(1)\mathrm{U}(n)),\nonumber\\
\mathbb{H}\mathrm{P}^n=\mathrm{Sp}(n+1)/\mathrm{Sp}(1)\mathrm{Sp}(n), &\mbox{and}&
\mathbb{O}\mathrm{P}^2=\mathrm{F}_4/\mathrm{Spin}(9).
\end{eqnarray}
Notice that on $\mathrm{S}^6=\mathrm{G}_2/\mathrm{SU}(3)$ and $\mathrm{S}^7=\mathrm{Spin}(7)/\mathrm{G}_2$, any invariant Finsler metrics on them must be Riemannian symmetric.

More generally, we have

\begin{proposition}\label{prop-classification-local}
Any compact connected Riemannian manifold locally isometric to a
compact rank-one symmetric space and satisfying Assumption
(I) must be one of the Riemannian rank-one symmetric spaces, i.e. for all $n>1$, $\mathrm{S}^n$, $\mathbb{R}\mathrm{P}^n$, $\mathbb{C}\mathrm{P}^n$,
$\mathbb{H}\mathrm{P}^n$, and $\mathbb{O}\mathrm{P}^2$.
\end{proposition}

\begin{proof}
Assume that $M$ is a locally isometric to a compact Riemannian rank-one symmetric space, non-simply connected, and satisfies Assumption (I).

We have a locally isometric finite covering map $\pi:\tilde{M}\rightarrow M$
such that $\tilde{M}$ is one of (\ref{list-cross}). All geodesics on $M$ are closed, and all prime closed geodesics on $M$ have the same length. So any geodesic passing $x\in \tilde{M}$ must contain points
in $\pi^{-1}(\pi(x))\backslash\{x\}\neq\emptyset$.

When $\tilde{M}=\mathrm{S}^n$ with $n>1$, $\pi^{-1}(x)\backslash\{x\}$ must coincide with the antipodal point $x'$ of $x$, otherwise $\pi^{-1}(\pi(x))$ is an infinite set,
which is a contradiction. So in this case $M$ is a Riemannian symmetric $\mathbb{R}\mathrm{P}^n$.

When $\tilde{M}=\mathbb{K}\mathrm{P}^n$ with $\mathbb{K}=\mathbb{C}$, $\mathbb{H}$ and $n\geq 2$ or $\mathbb{K}=\mathbb{O}$ and $n=2$,
for any $x\in M$ and any $x'$ in the cut locus of $x$, we can
find a sphere $M'=\mathbb{K}\mathrm{P}^1$ of constant curvature, such that $M'$ is
imbedded in $\tilde{M}$ as a totally geodesic sub-manifold and $x$ and $x'$ are contained in $M'$ as a pair of antipodal points.
As for the previous case, Assumption (I) and the finiteness of
$\pi^{-1}(\pi(x))$ implies $\pi^{-1}(\pi(x))$ must contain any point $x'$ in the cut locus of $x$. This is a contradiction to the finiteness of
$\pi^{-1}(\pi(x))$.

This ends of proof of this proposition.
\end{proof}

{\bf Example 2.} The flat torus $(\mathrm{T}^n,F)$ with $n>1$ does not
satisfy Assumption (I).

We can find arbitrarily long prime closed
geodesics on a flat torus. But the prime closed geodesics in the
same orbit must have the same length.

{\bf Example 3.} Any $\mathrm{SO}(n_1+1)\times
\mathrm{SO}(n_2+1)$-invariant Finsler metric
$F$ on $M=\mathrm{S}^{n_1}\times \mathrm{S}^{n_2}=\mathrm{SO}(n_1+1)\times \mathrm{SO}(n_2+1)/\mathrm{SO}(n_1)\times \mathrm{SO}(n_2)$ does
not satisfy Assumption (I).

The $\mathrm{SO}(n_1+1)\times \mathrm{SO}(n_2+1)$-invariance of $F$ may imply more isometries. If $n_i>1$, the term $\mathrm{SO}(n_i+1)$ can be changed to
$\mathrm{O}(n_i+1)$.

For each $i=1$ and $2$, we take $G'_i=\{e\}$ when $n_i=1$, $G'_i=\mathrm{SO}(n_i-1)$ when $n_i>3$, and $G'_i=\{g_i,e\}\subset \mathrm{O}(n_i+1)$
where $g$ is a reflection when $n_i=3$.
Then the fixed point set
$\mathrm{Fix}(G'_1\times G'_2, M)$ is a flat 2-dimensional torus
imbedded in $M$ as a totally geodesic submanifold.
Our claim is then implied by Example 2.


{\bf Example 4.} When $F$ is a left invariant Finsler metric on $G=\mathrm{SO}(2)\times \mathrm{SU}(2)$, $(G,F)$ does not satisfy Assumption (I).

This is a special case of the following proposition.

\begin{proposition} \label{prop-rank-bigger-than-1}
Any left invariant Finsler metric $F$
on a compact connected Lie group $G$ with $\mathrm{rk}G>1$ does not
satisfy Assumption (I). Further more, it has arbitrarily long
prime closed geodesics.
\end{proposition}

\begin{proof}
We fix a bi-invariant inner product
$|\cdot|_\mathrm{bi}^2=\langle\cdot,\cdot\rangle_\mathrm{bi}$
on $\mathfrak{g}=\mathrm{Lie}(G)$, and
a Cartan subalgebra $\mathfrak{t}$ in
$\mathfrak{g}$.
The Killing vector fields on $(G,F)$ are right invariant vector
fields.
There exists a positive constant $c>0$, such that
for any nonzero
vector $u\in\mathfrak{g}$, $F(X^u(g))>c$ for all $g\in G$.
The flow
generated by $X^u$ has the same period (which may be infinity) everywhere on $G$.

We can find a sequence of nonzero vectors $u_n\in \mathfrak{t}$ satisfying the following:
\begin{description}
\item{\rm (1)} Each $u_n$ generates an $\mathrm{S}^1$-subgroup.
\item{\rm (2)} The period $T_{n}$ for the flow $\rho_{n,t}$ generated by $X^{u_n}$ diverges to infinity.
\end{description}
Then the integration curve of $X^{u_n}$ at the maximum point
of $f(\cdot)=F(X^{u_n}(\cdot))$ provides a prime closed geodesic
$\gamma_n$. The length $l(\gamma_n)$
of $\gamma_n$ is at least $c\cdot T_{n}$ which diverges to infinity.

This ends the proof of the proposition.
\end{proof}

{\bf Example 5.} When $F$ is an $\mathrm{SO}(3)\times \mathrm{SU}(2)$-invariant
Finsler metric on $M=S^2\times S^3=\mathrm{SO}(3)\times \mathrm{SU}(2)/\mathrm{SO}(2)\times\{e\}$, it does not satisfy Assumption (I).

The metric $F$ is in fact
$\mathrm{O}(3)\times \mathrm{SU}(2)$-invariant. The fixed point set $M'=\mathrm{Fix}(g,M)$ of
a reflection $g\in \mathrm{O}(3)$ in $M$ can be identified as the Lie group
$G=\mathrm{SO}(2)\times \mathrm{SU}(2)$ such that $F|_{M'}$ is left invariant. By the discussion for Example 4, $(M,F)$ does not satisfy Assumption (I).

Based on Example 2-5, we can find more examples by the following
two propositions.

\begin{proposition} \label{prop-finite-cover}
If a compact connected homogeneous Finsler space $(M,F)$ is finitely covered by and locally isometric to one of Example 2-5, then $(M,F)$ does not
satisfy Assumption (I). In particular, it has arbitrarily long
prime closed geodesics.
\end{proposition}

\begin{proof}
Assume the Finsler space $(M,F)$ has a finite cover $(\tilde{M},\tilde{F})$
which is one of Example 2-5, such that the covering map
$\pi:(\tilde{M},\tilde{F})\rightarrow(M,F)$ is locally isometric.

On $(\tilde{M},\tilde{F})$, we can find a sequence of prime closed
geodesics $\tilde{\gamma}_n$ such that their lengths
$l(\tilde{\gamma}_n)$ diverge to infinity. Then
$\pi(\tilde{\gamma}_n)$ is a closed geodesic on $(M,F)$, which
corresponds to a prime closed geodesic $\gamma_n$. If $\pi$ is
an $m$-fold covering, then the lengths $l(\gamma_n)\geq l(\tilde{\gamma_n})/m$ which also diverge to infinity. So $(M,F)$ does not satisfy Assumption (I), which proves the proposition.
\end{proof}

\begin{proposition} \label{prop-rank-bigger-than-one-more}
Assume $(M,F)=(G/H,F)$ is a compact connected homogeneous Finsler space such that $G$ is a compact connected Lie group and $\mathrm{rk}{G}>
\mathrm{rk}{H}+1$, then $(M,F)$ does not satisfy Assumption (I), and it has arbitrarily long
prime closed geodesics.
\end{proposition}

\begin{proof}
Let $T$ be a maximal torus in $H$ and denote $M'=\mathrm{Fix}_o(T,M)$.
Then $M'$ is
totally geodesic in $M$ and $(M',F|_{M'})$ is finitely covered by and locally isometric to a compact connected Lie group $G'$ with $\mathrm{rk}G'>1$ and a left invariant Finsler metric $F'$.
The argument in the proofs of Proposition \ref{prop-rank-bigger-than-1} and Proposition \ref{prop-finite-cover} shows there exist arbitrarily long prime
closed geodesics on $(G',F')$ as well as on $(M',F|_{M'})$ and $(M,F)$.

This ends the proof of the proposition.
\end{proof}
\section{Algebraic setup and $g_u^F$-orthogonality}

Assume $(M,F)=(G/H,F)$, where $G=I_0(M,F)$, is a compact connected homogeneous Finsler space satisfying Assumption (I), i.e. it
has only one orbit of prime closed geodesics.

We fix an $\mathrm{Ad}(G)$-invariant inner product $|\cdot|_\mathrm{bi}^2=\langle\cdot,\cdot\rangle_\mathrm{bi}$.
For simplicity, we call the orthogonality with respect to this inner product {\it bi-invariant}. There is a unique bi-invariant
reductive decomposition $\mathfrak{g}=\mathfrak{h}+\mathfrak{m}$.
We denote $\mathrm{pr}_\mathfrak{h}$ and $\mathrm{pr}_\mathfrak{m}$ the projection maps according
to this decomposition.

Any vector $u\in\mathfrak{m}$ represents a tangent vector in
$T_oM$. Meanwhile $u\in\mathfrak{g}$ also defines the
Killing vector field $X^u$ which satisfies $X^u(o)=u\in T_oM=\mathfrak{m}$.
Any Killing vector field $X$ of $(M,F)$ satisfying
$X(o)=u$ can be presented as $X^{u+u'}$ for some $u'\in\mathfrak{h}$. By Proposition 3.4 in \cite{YD-2014}, the integration curve of $X^{u+u'}$
at $o$ is a geodesic if and only if
$$\langle u,[u+u',\mathfrak{m}]_\mathfrak{m}\rangle^F_u=0.$$
In particular, when $u'=0$ we have

\begin{lemma} \label{lemma-6}
Let $(G/H,F)$ be a compact connected homogeneous Finsler space with a bi-invariant reductive
decomposition $\mathfrak{g}=\mathfrak{h}+\mathfrak{m}$.
Assume $u\in\mathfrak{m}$ is a nonzero vector
satisfying
\begin{equation}\label{equ-eta-0}
\langle u,[u,\mathfrak{m}]_\mathfrak{m}\rangle^F_u=0,
\end{equation}
then the integration curve of $X^u$ at $o$ is a geodesic. In particular, if $u$ generates an $\mathrm{S}^1$-subgroup in $G$, then $X$ generates a closed
geodesic at $o$.
\end{lemma}

Define $\mathcal{C}$ to be the subset of all
$u\in\mathfrak{m}$ such that $|u|_\mathrm{bi}=1$ and
the geodesic of $(M,F)$ passing $o$ in the direction
of $u$ is closed. Assumption (I) implies $\mathcal{C}$
is an $\mathrm{Ad}(H)$-orbit.

Lemma \ref{lemma-6} is a key technique for us to determine
$\mathcal{C}$. Vectors in $\mathfrak{m}$ which generate $\mathrm{S}^1$-subgroups are not hard to be found, for example,
from $\mathfrak{t}\cap\mathfrak{m}$ or a root plane.
The equality (\ref{equ-eta-0}) implies the $g_u^F$-orthogonality
is the remaining issue to be considered, which will be discussed
for the cases $\dim M$ is even and odd separately.

{\bf Case 1.} Assume $\dim M$ is even.

By Theorem \ref{thm-rank-inequality}, we can find a Cartan subalgebra $\mathfrak{t}$ of $\mathfrak{g}$ which is contained in $\mathfrak{h}$. With respect to $\mathfrak{t}$, we have the
root plane decomposition for $\mathfrak{g}$,
\begin{eqnarray}\label{root-plane-decomp-g}
\mathfrak{g}=\mathfrak{t}+\sum_{\alpha\in\Delta}
\mathfrak{g}_{\pm\alpha},
\end{eqnarray}
where $\Delta\subset\mathfrak{t}^*$ is the root system of $\mathfrak{g}$, and
$\mathfrak{g}_{\pm\alpha}$ is the root plane.

For $\mathfrak{h}$, we have a similar root plane decomposition with respect to $\mathfrak{t}$. The root system $\Delta'$ of $\mathfrak{h}$ is a
subset of $\Delta$, and $\mathfrak{h}_{\pm\alpha}
=\mathfrak{g}_{\pm\alpha}$ when $\alpha\in\Delta'$.
Because the reductive
decomposition is bi-invariant, each root plane $\mathfrak{g}_{\pm\alpha}$ is contained in either $\mathfrak{h}$
or $\mathfrak{m}$.

For any nonzero vector $u\in\mathfrak{g}_{\pm\alpha}\subset\mathfrak{m}$,
we have the following $g_u^F$-orthogonality \cite{XDHH}.

\begin{lemma}\label{ortho-lemma-0}
Let $G/H$ be an even dimensional compact connected homogeneous Finsler space,
and keep all relevant notations and assumptions. Then for any nonzero $u\in\mathfrak{g}_{\pm\alpha}\subset\mathfrak{m}$, $\mathfrak{g}_{\pm\alpha}$ is $g_u^F$-orthogonal
to all other root planes in $\mathfrak{m}$.
\end{lemma}

Sketchily, Lemma \ref{ortho-lemma-0} can be proved as following.
Let $T'=\exp\ker\alpha$ be the sub-torus in $H$ generated by
the kernel of $\alpha$ in $\mathfrak{t}$. We have a $\mathrm{Ad}(T
)$-invariant decomposition
of $\mathfrak{m}$ such that each summand is a sum of root planes and corresponds to a different irreducible representation of $T'$. Because the $T'$-action preserves $u\in T_oM$, the $\mathrm{Ad}(T')$-action
preserves the inner product $g_u^F$ on $\mathfrak{m}$. Schur's Lemma
implies different summand in $\mathfrak{m}$ are $g_u^F$-orthogonal
to each other. In particular, the summand in $\mathfrak{m}$ corresponding to the trivial representation is $\mathfrak{g}_{\pm\alpha}$, which is then $g_u^F$-orthogonal to all other root planes in $\mathfrak{m}$. See Lemma 5.3 in \cite{XDHH} and its proof for more details.

The same thought can be applied to the case that $\dim M$ is odd,
concerning different decompositions and different group actions, and
providing more $g_u^F$-orthogonality.

{\bf Case 2.} Assume $\dim M$ is odd.

By Theorem \ref{thm-rank-inequality}, we can find a Cartan subalgebra $\mathfrak{t}$ of $\mathfrak{g}$ such that $\mathfrak{t}\cap\mathfrak{h}$ is a Cartan subalgebra of $\mathfrak{h}$ and $\dim\mathfrak{t}\cap\mathfrak{m}=1$. For
simplicity, we call this $\mathfrak{t}$ a {\it fundamental Cartan
subalgebra for} $G/H$.

With respect to $\mathfrak{t}$ and $\mathfrak{t}\cap\mathfrak{h}$,
we have the root plane decompositions,
\begin{eqnarray*}
\mathfrak{g}=\mathfrak{t}+\sum_{\alpha\in\Delta}
\mathfrak{g}_{\pm\alpha},\quad\mbox{and}\quad
\mathfrak{h}=\mathfrak{t}\cap\mathfrak{h}+\sum_{\alpha'\in\Delta'}
\mathfrak{h}_{\pm\alpha'},
\end{eqnarray*}
where $\Delta$ and $\Delta'$ are root systems for $\mathfrak{g}$
and $\mathfrak{h}$ respective.

For the simplicity of notations, we apply the convention in
\cite{XD-towards-odd}\cite{XW-KVCL-Riemannian-normal}\cite{XZ-2017}.
Using the inner product $\langle\cdot,\cdot\rangle_{\mathrm{bi}}$,
we identify roots of $\mathfrak{g}$ as vectors in $\mathfrak{t}$,
and using the restriction of $\langle\cdot,\cdot\rangle_\mathrm{bi}$
to $\mathfrak{h}$, we identify roots of $\mathfrak{h}$ as vectors in $\mathfrak{t}\cap\mathfrak{h}$. Then the root systems $\Delta$
and $\Delta'$ are subsets in $\mathfrak{t}$ and $\mathfrak{t}\cap\mathfrak{h}$ rather than $\mathfrak{t}^*$ and
$(\mathfrak{t}\cap\mathfrak{h})^*$ respectively.

Assume $T_H$ is the maximal torus in $H$, generated by $\mathfrak{t}\cap\mathfrak{h}$. With respect to different
irreducible $T_H$-representations, $\mathfrak{g}$ can be
decomposed as
\begin{equation}
\mathfrak{g}=\hat{\mathfrak{g}}_0+
\sum_{0\neq\alpha'\in\mathfrak{t}\cap\mathfrak{h}}
\hat{\mathfrak{g}}_{\pm\alpha'}=\mathfrak{c}_\mathfrak{g}
(\mathfrak{t}\cap\mathfrak{h})+
\sum_{0\neq\alpha'\in\mathfrak{t}\cap\mathfrak{g}}
\left(\sum_{\mathrm{pr}_\mathfrak{h}(\alpha)=\alpha'}
\mathfrak{g}_{\pm\alpha}\right).
\end{equation}

For each $\alpha'\in\mathfrak{t}\cap\mathfrak{h}$,
$\hat{\mathfrak{g}}_{\pm\alpha'}=
\hat{\mathfrak{g}}_{\pm\alpha'}\cap\mathfrak{h}+
\hat{\mathfrak{g}}_{\pm\alpha'}\cap\mathfrak{m}$.
In particular
$\hat{\mathfrak{g}}_{0}\cap\mathfrak{h}=
\mathfrak{t}\cap\mathfrak{h}$,
$\hat{\mathfrak{g}}_{\pm\alpha'}\cap\mathfrak{h}=0$
 if and only if $\alpha'\notin\Delta'$, and
$\hat{\mathfrak{g}}_{\pm\alpha'}\cap\mathfrak{h}=\mathfrak{h}_{\pm\alpha'}\neq0$
 if and only if $\alpha'\in\Delta'$.

Moreover, we have the $\mathrm{Ad}(T_H)$-invariant
decomposition
\begin{equation}\label{T-H-decomp-m}
\mathfrak{m}=\sum_{\alpha'\in\mathfrak{t}\cap\mathfrak{h}}
\hat{\mathfrak{m}}_{\pm\alpha'},
\end{equation}
where $\hat{\mathfrak{m}}_{\pm\alpha'}=\hat{\mathfrak{g}}_{\pm\alpha'}
\cap\mathfrak{m}$, and in particular,
$\hat{\mathfrak{m}}_{0}$ is a subalgebra, the centralizer
$\mathfrak{c}_\mathfrak{m}(\mathfrak{h})$ of $\mathfrak{h}$ in
$\mathfrak{m}$. Either $\hat{\mathfrak{m}}_0=\mathfrak{t}+\mathfrak{g}_{\pm\alpha}$
when there exists a root $\alpha$ in $\mathfrak{t}\cap\mathfrak{m}$, or $\hat{\mathfrak{m}}_0=\mathfrak{t}\cap\mathfrak{m}$ otherwise.

The $T_H$-action preserves $u\in T_oM$ for any nonzero
$u\in\hat{\mathfrak{m}}_0$. So applying the similar method as
for Lemma \ref{ortho-lemma-0}, we get

\begin{lemma} \label{ortho-lemma-1}
Assume $(G/H,F)$ is an odd dimensional compact connected homogeneous Finsler space
and keep all relevant notations and assumptions. Then
for any nonzero $u\in\hat{\mathfrak{m}}_0$,
the decomposition (\ref{T-H-decomp-m}) is $g_u^F$-orthogonal.
\end{lemma}

When $\dim\hat{\mathfrak{m}}_0=3$, we need more $g_u^F$-orthogonality which requires suitable choices of $u$
as following. There exist two different vectors $u_1$ and $u_2$ in $\hat{\mathfrak{m}}_0$ such that
$|u_1|_\mathrm{bi}=|u_2|_\mathrm{bi}=1$,
$F(u_1)=\min\{F(u)| u\in\hat{\mathfrak{m}}_0, |u|_\mathrm{bi}=1\}$, and
$F(u_2)=\max\{F(u)| u\in\hat{\mathfrak{m}}_0, |u|_\mathrm{bi}=1\}$.

Then we have
\begin{lemma}\label{ortho-lemma-1-5}
Assume $(G/H,F)$ is an odd dimensional compact connected homogeneous Finsler space with $\dim\hat{\mathfrak{m}}_0=3$, and keep all relevant notations
and assumptions, then for $u_i$ chosen above, we have
$\langle u_i,[u_i,\hat{\mathfrak{m}}_0]\rangle^F_{u_i}=0$.
\end{lemma}

Assume $\alpha$ is a root of $\mathfrak{g}$ such that
$\mathrm{pr}_\mathfrak{h}(\alpha)=\alpha'\neq 0$. Denote
$\mathfrak{t}'$ the bi-invariant orthogonal complement of $\alpha'$ in $\mathfrak{t}\cap\mathfrak{h}$, then we have
a sub-torus $T'\subset T_H\subset H$ generated by
$\mathfrak{t}'$. Denote $\mathrm{pr}_{\mathfrak{t}'}$
the bi-invariant orthogonal projection to $\mathfrak{t}'$.

According to different irreducible $T'$-representations,
we have the $\mathrm{Ad}(T')$-invariant decomposition
\begin{equation}\label{T'-decomp-m}
\mathfrak{m}=\sum_{\beta''\in\mathfrak{t}'}
\hat{\hat{\mathfrak{m}}}_{\pm\beta''}=
\sum_{\beta''\in\mathfrak{t}'}\left(
\sum_{\beta'\in\mathfrak{t}\cap\mathfrak{h},
\mathrm{pr}_{\mathfrak{t}'}(\beta')=\beta''}
\hat{\mathfrak{m}}_{\pm\beta'}\right).
\end{equation}

For any nonzero
$u\in\hat{\hat{\mathfrak{m}}}_0$, the $T'$-action preserves $u\in T_oM$. Applying the similar method for Lemma \ref{ortho-lemma-0},
we get

\begin{lemma}\label{ortho-lemma-2}
Assume $(G/H,F)$ is an odd dimensional compact connected homogeneous Finsler space
and keep all relevant notations and assumptions. Then for any
nonzero $u\in\hat{\hat{\mathfrak{m}}}_0$, the decomposition
(\ref{T'-decomp-m}) is $g_u^F$-orthogonal.
\end{lemma}

We will need more $g_u^F$-orthogonality inside $\hat{\hat{\mathfrak{m}}}_0=\sum_{t\geq0}
\hat{\mathfrak{m}}_{\pm t\alpha'}$, which may be achieved with the reversibility assumption for $F$. We can find an element $g\in T_H$ which action on each root plane in
$\hat{\mathfrak{g}}_{\pm t\alpha'}$ is a rotation with angle $t\pi$. Assume the invariant metric $F$ is reversible, then for any nonzero $u\in\hat{\mathfrak{m}}_{\pm\alpha'}$, $w_1,w_2\in\mathfrak{m}$,
\begin{eqnarray*}
\langle w_1,w_2\rangle^F_u&=&\langle \mathrm{Ad}(g)w_1,\mathrm{Ad}(g)w_2\rangle^F_{\mathrm{Ad}(g)u}
=\langle \mathrm{Ad}(g)w_1,\mathrm{Ad}(g)w_2\rangle^F_{-u}\\
&=&\langle \mathrm{Ad}(g)w_1,\mathrm{Ad}(g)w_2\rangle^F_{u},
\end{eqnarray*}
i.e. the $\mathrm{Ad}(g)$-action preserves $g_u^F$.
Applying the similar method as for Lemma \ref{ortho-lemma-0}
for the action of the group generated by $T'$ and $g$, we get
\begin{lemma}\label{ortho-lemma-2-5}
Assume $(G/H,F)$ is an odd dimensional compact connected reversible homogeneous Finsler space and we keep all relevant notations and assumptions.
Then for any nonzero $u\in\hat{\mathfrak{m}}_{\pm\alpha'}$,
$\hat{\mathfrak{m}}_{\pm t_1\alpha'}$ and
$\hat{\mathfrak{m}}_{\pm t_2\alpha'}$ are $g_u^F$-orthogonal
when $t_1$ and $t_2$ are non-negative, and $t_1-t_2\in 2\mathbb{Z}$. In particular, we have
$\langle\hat{\mathfrak{m}}_{\pm\alpha'},\hat{\mathfrak{m}}_0
\rangle^F_u=0$.
\end{lemma}

For any nonzero
$u\in\hat{\mathfrak{m}}_{\pm\alpha'}$ where $\alpha'\in\mathfrak{t}\cap\mathfrak{h}$ is nonzero,
we have a plane $\mathbb{R}u+[\mathfrak{t}\cap\mathfrak{h},u]$ in
$\mathfrak{m}$. The restriction of $F$ to this plane coincides
with $|\cdot|_\mathrm{bi}$ up to a scalar multiplication. So we have

\begin{lemma}\label{ortho-lemma-2-6}
Assume $(G/H,F)$ is an odd dimensional compact connected homogeneous Finsler space and keep all relevant notations and assumptions. Then for any nonzero $u\in\hat{\mathfrak{m}}_{\pm\alpha'}$ such that $\alpha'$ is a nonzero vector in $\mathfrak{t}\cap\mathfrak{h}$, we have
$\langle u,[\mathfrak{t}\cap\mathfrak{h},u]\rangle^F_u=0$.
In particular, when $u\in\mathfrak{g}_{\pm\alpha}\subset\hat{\mathfrak{m}}_{\pm\alpha'}$,
we have
$\langle u,[\mathfrak{t}\cap\mathfrak{m},u]\rangle^F_u=0$.
\end{lemma}

At the end of this section, we remark that the  $g_u^F$-orthogonality lemmas in Case 2 are reformulations
of Lemma 3.6, Lemma 3.7 and Lemma 3.8 in \cite{XD-towards-odd},
where more details can be found.
\label{Section-algebraic-setup-ortho}
\section{Classification when $\dim M$ is even}

In this section, we assume $(M,F)=(G/H,F)$ with $G=I_0(M,F)$ is an even dimensional compact connected homogeneous Finsler space satisfying Assumption (I), i.e. it has only one orbit of prime closed geodesics. We keep all relevant notations and assumptions in Case 1, Section \ref{Section-algebraic-setup-ortho}. Recall that $\mathcal{C}$ is the subset of all vectors $u\in\mathfrak{m}$ satisfying $|u|_\mathrm{bi}=1$ and the geodesic passing $o$ in the direction of $u\in T_oM$ is closed. Because of Assumption (I),
$\mathcal{C}$ is an $\mathrm{Ad}(H)$-orbit.

Our goal is to prove the following classification theorem.

\begin{theorem}\label{classification-thm-1}
Any even dimensional compact connected homogeneous Finsler space
with only one orbit of prime closed geodesics is a compact rank-one Riemannian symmetric space, i.e. one of the following, $\mathrm{S}^{2n}$,
$\mathbb{R}\mathrm{P}^{2n}$
$\mathbb{C}\mathrm{P}^n$,
$\mathbb{H}\mathrm{P}^n$,
and $\mathbb{O}\mathrm{P}^2$.
\end{theorem}

Its proof relies on the following lemmas.

\begin{lemma} \label{lemma-classification-even-dim-1}
Assume $(G/H,F)$ is an even dimensional compact connected homogeneous Finsler space
and keep all relevant notations and assumptions. Then
any $u\in\mathfrak{g}_{\pm\alpha}\subset\mathfrak{m}$ such that $|u|_{\mathrm{bi}}=1$ is contained in $\mathcal{C}$.
\end{lemma}

\begin{proof}
The vector $u$ indicated in the lemma is contained
in a subalgebra of type $\mathrm{A}_1$, so it generates an $\mathrm{S}^1$-subgroup
in $G=I_0(M,F)$. By Lemma \ref{ortho-lemma-0},
$$\langle u,[u,\mathfrak{m}]_\mathfrak{m}\rangle^F_u
\subset\langle \mathfrak{g}_{\pm\alpha},
\sum_{\beta\neq\pm\alpha,\beta\notin\Delta'}\mathfrak{g}_{\pm\beta}
\rangle^F_u=0.$$
Then Lemma \ref{lemma-6} indicates that $u\in\mathcal{C}$ when
$|u|_{\mathrm{bi}}=1$.
\end{proof}

As a consequence of Lemma \ref{lemma-classification-even-dim-1}, we have
\begin{lemma}\label{lemma-classification-even-dim-2}
Assume $(M,F)=(G/H,F)$ is an even dimensional compact connected homogeneous Finsler space satisfying Assumption (I). Then the isotropy representation of $G/H$ is irreducible.
\end{lemma}
\begin{proof}
Assume conversely that the isotropy representation is not irreducible, i.e. there exists a non-trivial $\mathrm{Ad}(H)$-invariant decomposition
$\mathfrak{m}=\mathfrak{m}_1+\mathfrak{m}_2$.
Because the Cartan subalgebra $\mathfrak{t}$ is contained
in $\mathfrak{h}=\mathrm{Lie}(H)$, each $\mathfrak{m}_i$  is a sum of root planes. So we can find a root $\alpha_i$
with $\mathfrak{g}_{\pm\alpha_i}\subset\mathfrak{m}_i$, and the
vector $u_i\in\mathfrak{g}_{\pm\alpha_i}\cap\mathcal{C}$ by Lemma \ref{lemma-classification-even-dim-1}, for each $i=1$ and $2$ respectively. But it is impossible because the orbits
$\mathcal{C}=\mathrm{Ad}(H)u_i\subset\mathfrak{m}_i$ for $i=1$ and $2$ do not intersect with itself.
\end{proof}

\begin{lemma}\label{lemma-classification-even-3}
Assume $(G/H,F)$ is an even dimensional compact connected homogeneous Finsler space
satisfying Assumption (I) and keep all relevant notations and assumptions. Then there do not exist a pair of linearly independent roots $\alpha$
and $\beta$ such that $\mathfrak{g}_{\pm\alpha}$ and $\mathfrak{g}_{\pm\beta}$ are contained in $\mathfrak{m}$, and $\alpha\pm\beta$ are not roots of $\mathfrak{g}$.
\end{lemma}

\begin{proof}
Assume conversely that there exist a pair of
roots $\alpha$ and $\beta$ indicated in this lemma.

Inside the maximal torus $T=\exp\mathfrak{t}$,
we have a codimension two sub-torus $T'=\exp \mathfrak{t}'=\exp(\ker\alpha\cap\ker\beta)$.
The totally geodesic submanifold $(\mathrm{Fix}_o(T',M),F|_{\mathrm{Fix}_o(T',M)})$
is finitely covered
by and locally isometric to $M'=\mathrm{S}^2\times \mathrm{S}^2=\mathrm{SO}(3)\times \mathrm{SO}(3)/\mathrm{SO}(2)\times \mathrm{SO}(2)$ with
an $\mathrm{SO}(3)\times \mathrm{SO}(3)$-invariant Finsler metric.

According to Example 3 and Proposition
\ref{prop-finite-cover} in Section \ref{Section-homogeneous-examples}, we can find a sequence of prime closed geodesics
$\gamma_n$
on $(\mathrm{Fix}_o(T',M),F|_{\mathrm{Fix}_o(T',M)})$, as well
as on $(M,F)$, which lengths $l(\gamma_n)$'s diverge to infinity.
They can not belong to the same orbit, which is a contradiction
to Assumption (I).

This ends the proof of this lemma.
\end{proof}

Now we are ready to prove Theorem \ref{classification-thm-1}.

{\bf Proof of Theorem \ref{classification-thm-1}.}
Let $(M,F)=(G/H,F)$ with $G=I_0(M,F)$
be an even dimensional compact connected homogeneous space satisfying Assumption (I). Denote $H_0$ the identity component of the isotropy subgroup $H$.

By Lemma \ref{lemma-classification-even-3} and the algebraic discussion in Section 6 of \cite{Wa1972},
$(\mathfrak{g},\mathfrak{h})$ belongs to the Wallach's list, i.e. it is one of the following,
\begin{description}
\item{\rm (1)} $(\mathrm{B}_n,\mathrm{D}_n)$, $(\mathrm{A}_n,\mathbb{R}\oplus \mathrm{A}_{n-1})$,
$(\mathrm{C}_n,\mathrm{A}_1\oplus \mathrm{C}_{n-1})$, $(\mathrm{F}_4,\mathrm{B}_4)$.
\item{\rm (2)} $(\mathrm{A}_2,\mathbb{R}\oplus\mathbb{R})$, $(\mathrm{C}_3,\mathrm{A}_1\oplus \mathrm{A}_1\oplus \mathrm{A}_1)$, $(\mathrm{F}_4,\mathrm{D}_4)$.
\item{\rm (3)} $(\mathrm{G}_2,\mathrm{A}_2)$.
\item{\rm (4)} $(\mathrm{C}_n,\mathbb{R}\oplus \mathrm{C}_{n-1})$.
\end{description}

When the pair $(\mathfrak{g},\mathfrak{h})$ belongs to (1) or (3), $(G/H,F)$ is locally isometric to one of the compact rank-one Riemannian symmetric
spaces. The theorem follows Proposition \ref{prop-classification-local} immediately.

When $(\mathfrak{g},\mathfrak{h})$
belongs to (4), we have a unique $\mathrm{Ad}(H_0)$-invariant decomposition
$\mathfrak{m}=\mathfrak{m}_1+\mathfrak{m_2}$
where $\mathfrak{m}_i$'s are irreducible representations of $H_0$
with $\dim\mathfrak{m}_1=2$ and $\dim\mathfrak{m}_2=4$. Here $H_0$ is the identity component of $H$, which is a normal subgroup of $H$ as well.

For each $g\in H$, the $\mathrm{Ad}(g)$-action permutes the two $\mathfrak{m}_i$-factors in $\mathfrak{m}$.
Because the two factors have different dimensions, we must
have $\mathrm{Ad}(g)\mathfrak{m}_i=\mathfrak{m}_i$ for each $i\in\{1,2\}$ and each $g\in H$.
So the isotropy representation for $G/H$ is not irreducible,
which contradicts Lemma \ref{lemma-classification-even-dim-2}.

Finally, we consider the case that
$(\mathfrak{g},\mathfrak{h})$ belongs to (2).

In this case, $(M,F)$ is finitely
covered by $G/H_0=\mathrm{SU}(3)/T^2$, $\mathrm{Sp}(3)/\mathrm{Sp}(1)^3$ or $\mathrm{F}_4/\mathrm{Spin}(8)$. We have a unique
decomposition $\mathfrak{m}=\mathfrak{m}_1+\mathfrak{m}_2+\mathfrak{m}_3$,
in which all three $\mathfrak{m}_i$-factors are different irreducible representations of $H_0$ with the same dimension.

From $g\in H$ to the permutation action of $\mathrm{Ad}(g)$ on the
three $\mathfrak{m}_i$-factors defines an injection from $H/H_0$
to $S_3$, the permutation group of three elements. So the fundamental group $\pi_1(M)$ is an subgroup of $S_6$.

If $\pi_1(M)=\{e\}$ or $\mathbb{Z}_2$, the isotropy representation
is not irreducible, which contradicts Lemma \ref{lemma-classification-even-dim-2}. If $\pi_1(M)=S_3$, the prime closed geodesic representing a homotopy class of order 2 and another representing one of order 3 can not belong to the same orbit, which contradicts Assumption (I).
To summarize, we must have $H/H_0=\pi_1(\pi)=\mathbb{Z}_3$.

When $G/H_0=\mathrm{SU}(3)/T^2$, we can find the matrix
\begin{equation}\label{006}
g=\left(
      \begin{array}{ccc}
        0 & 1 & 0 \\
        0 & 0 & 1 \\
        1 & 0 & 0 \\
      \end{array}
    \right)\in \mathrm{SU}(3)
\end{equation}
in $H$. The reason is following. For any element $g'$ from
a suitable connected component of $H$, $\mathrm{Ad}(gg')$ preserves each $\mathfrak{m}_i$. Direct calculation shows that $gg'$ is a diagonal matrix, i.e. it is contained
in $H_0$. So we have $g\in H_0 g'=g'H_0\subset H$.

The centralizer $\mathfrak{c}(g)=\{u\in\mathfrak{g}\mbox{ with }\mathrm{Ad}(g)u=u\}$ for $g$ in (\ref{006}) is an Abelian subalgebra
spanned by
\begin{equation}\label{007}
u=\left(
    \begin{array}{ccc}
      0 & 1 & -1 \\
      -1 & 0 & 1 \\
      1 & -1 & 0 \\
    \end{array}
  \right)\quad\mbox{and}\quad
v=\sqrt{-1}\left(
  \begin{array}{ccc}
    0 & 1 & 1 \\
    1 & 0 & 1 \\
    1 & 1 & 0 \\
  \end{array}
\right)
\end{equation}
in $\mathfrak{m}$, and $\mathfrak{c}(g)\cap\mathfrak{h}=0$.
By the fixed pointed set technique in \cite{XZ-2017} (see Section 2)
$\mathrm{Fix}_o(g,M)$  is a flat torus. By Example 2 in Section \ref{Section-homogeneous-examples}, this is a contradiction to
Assumption (I).

When $G/H_0=\mathrm{Sp}(3)/\mathrm{Sp}(1)^3$, we can similarly argument that the matrix $g$
in (\ref{006}) is contained in $H$.
Then its centralizer $\mathfrak{c}(g)$ in $\mathfrak{g}=\mathrm{sp}(3)$ can be decomposed as a direct sum of ideals,
 $$\mathfrak{c}(g)=\mathfrak{g}_1\oplus\mathfrak{g}_2
 \oplus\mathbb{R}u,$$
where
$u\in\mathfrak{m}$ is the same as in (\ref{007}), and for each $i=1$ and 2,  $\mathfrak{g}_i=\mathrm{Im}\mathbb{H}v_i$ is a subalgebra of type $\mathrm{A}_1$, where
\begin{equation}
 v_1=\left(
  \begin{array}{ccc}
    1 & 1 & 1 \\
    1 & 1 & 1 \\
    1 & 1 & 1 \\
  \end{array}\right)
\quad \mbox{and} \quad
v_2=\left(
      \begin{array}{ccc}
        -2 & 1 & 1 \\
        1 & -2 & 1 \\
        1 & 1 & -2 \\
      \end{array}
    \right).
\end{equation}
Notice $\mathfrak{c}(g)\cap\mathfrak{h}=\mathrm{Im}\mathbb{H}I$ is also a subalgebra
of type $\mathrm{A}_1$, diagonally imbedded in $\mathfrak{g}_1\oplus\mathfrak{g}_2$.
So
 $(\mathrm{Fix}_o(g,M),F|_{\mathrm{Fix}_o(g,M)})$ is locally isometric to and finitely covered by $(M',F')$ where
$M'=\mathrm{S}^1\times \mathrm{S}^3=\mathrm{SO}(2)\times \mathrm{SO}(4)/e\times \mathrm{SO}(3)$ and $F'$ is $\mathrm{SO}(2)\times \mathrm{SO}(4)$-invariant.
By Example 3 in Section \ref{Section-homogeneous-examples} and Proposition \ref{prop-finite-cover}, this is a contradiction
to Assumption (I).

When $G/H_0=\mathrm{F}_4/\mathrm{Spin}(8)$, we view $\mathrm{F}_4$ as the automorphism group
of the exceptional Jordan algebra $\mathfrak{h}_3(\mathbb{O})$, the space of all Hermitian $3\times3$ Octonian matrices with the product $\circ$ defined by $A\circ B=\frac12(AB+BA)$.
Then $H_0=\mathrm{Spin}(8)$ can be identified as the subgroup fixing
$X_1=\mathrm{diag}(1,0,0)$, $X_2=\mathrm{diag}(0,1,0)$ and
$X_3=\mathrm{diag}(0,0,1)$. See Chapter 16 in \cite{Ad1996}
for more details.

The automorphism $g$ of $\mathfrak{h}_3(\mathbb{O})$
given by the conjugation by the matrix in (\ref{007})
is an element in $\mathrm{F}_4$ of order 3, such that $\mathrm{Ad}(g)$ induces an
outer automorphism of $\mathfrak{h}$. Because $H/H_0=\mathbb{Z}_3\subset S_3=\mathrm{Out}(\mathfrak{h})$, we see that
$g$ is contained in $H$.

For $g\in H$ described above,
its centralizer $\mathfrak{c}(g)$ in $\mathfrak{g}$ is a $22$-dimensional subalgebra of rank 4, and $\mathfrak{c}\cap\mathfrak{h}=\mathrm{G}_2$. By the fixed point set technique and Proposition \ref{prop-rank-bigger-than-one-more},
there exist arbitrarily long prime closed geodesic on
$(\mathrm{Fix}_o(g,M),F|_{\mathrm{Fix}_o(g,M)})$ as well as on
$(M,F)$. This is a contradiction to Assumption (I).

In fact, Proposition \ref{prop-rank-bigger-than-one-more} can
also be applied to the cases that $G/H_0=\mathrm{SU}(3)/T^2$ and
$\mathrm{Sp}(3)/\mathrm{Sp}(1)^3$. If we want to explicitly describe $\mathrm{Fix}_o(g,M)$ when $(\mathfrak{g},\mathfrak{h})=
(\mathrm{F}_4,\mathrm{D}_4)$,
we will find that
$\mathrm{Fix}_o(g,M)$ is finitely covered by and locally isometric
to the homogeneous Finsler space
$(M',F')=(S^1\times S^7,F')=(\mathrm{SO}(2)\times \mathrm{Spin}(7)/e\times \mathrm{G}_2,F')$.
The metric $F'$ is in fact $\mathrm{SO}(2)\times \mathrm{SO}(8)$-invariant, so
we can alternatively get the contradiction by Example 3 and
Proposition \ref{prop-finite-cover} in Section \ref{Section-homogeneous-examples}.

To summarize, by case by case discussion, we have proved Theorem
\ref{classification-thm-1}.

\section{Classification theorem and key lemmas when $\dim M$ is odd}

In the following two sections, we will study the odd dimensional
case and prove the following theorem.

\begin{theorem}\label{classification-thm-2}
Any odd dimensional compact connected reversible homogeneous
Finsler space with only one orbit of prime closed geodesics
must be a Riemannian symmetric $\mathrm{S}^n$ or a Riemannian symmetric $\mathbb{R}\mathrm{P}^n$.
\end{theorem}

We keep
all relevant notations and assumptions in Case 2, Section \ref{Section-algebraic-setup-ortho}. Using the bi-invariant inner product on $\mathfrak{g}$ and its restriction to $\mathfrak{h}$, we identify roots of $\mathfrak{g}$ and $\mathfrak{h}$ as vectors
in $\mathfrak{t}$ and $\mathfrak{t}\cap\mathfrak{h}$ respectively.

We still denote $\mathcal{C}$
the subset of all vectors $u\in\mathfrak{m}$ such that
$|u|_\mathrm{bi}=1$ and the geodesic passing $o$ in the direction of $u$ is closed. When Assumption (I) is satisfied, $\mathcal{C}$ is an $\mathrm{Ad}(H)$-orbit.

For preparation, we need the following lemmas.

\begin{lemma}\label{key-lemma-1}
Assume $(G/H,F)$ is an odd dimensional compact connected homogeneous Finsler space satisfying $\mathrm{rk}G=\mathrm{rk}H+1$,
and keep all relevant notations and assumptions. Then we can find
two different vectors $u_1$ and $u_2$ in
$\hat{\mathfrak{m}}_{0}\cap\mathcal{C}$.
\end{lemma}

\begin{proof}
There are two cases to be considered, $\dim\hat{\mathfrak{m}}_0=1$
or $\dim\hat{\mathfrak{m}}_0=3$.

Assume $\dim\hat{\mathfrak{m}}_0=1$, i.e. $\hat{\mathfrak{m}}_0=\mathfrak{t}\cap\mathfrak{m}$.
We choose $u_1\in\mathfrak{t}\cap\mathfrak{m}$ such that
$|u_1|_\mathrm{bi}=1$. By Lemma \ref{ortho-lemma-1},
\begin{eqnarray*}
\langle u_1,[u_1,\mathfrak{m}]_\mathfrak{m}\rangle^F_{u_1}
\subset\langle \hat{\mathfrak{m}}_0,
\sum_{\alpha'\neq0}\hat{\mathfrak{m}}_{\pm\alpha'}\rangle^F_u=0.
\end{eqnarray*}
On the other hand, $\mathfrak{t}\cap\mathfrak{m}$ is bi-invariant
orthogonal to $\mathfrak{t}\cap\mathfrak{h}$ which generates the
torus $T_H$. So $u\in\mathfrak{t}\cap\mathfrak{m}$ generates
an $\mathrm{S}^1$-subgroup. By Lemma \ref{lemma-6}, we get $u_1\in\mathcal{C}$.

By the same argument, we can also get $u_2=-u_1\in\mathcal{C}$.

This proves the case when $\dim\hat{\mathfrak{m}}_0=1$.

Assume $\dim\hat{\mathfrak{m}}_0=3$. By Lemma \ref{ortho-lemma-1-5}, we can find two different vectors $u_1$
and $u_2$ in $\hat{\mathfrak{m}}_0$, such that
$|u_1|_\mathrm{bi}=|u_2|_\mathrm{bi}=1$, and
\begin{equation}\label{005}
\langle u_i,[u_i,\hat{\mathfrak{m}}_0]\rangle^F_{u_i}=0.
\end{equation}
For each $i$,
\begin{eqnarray*}
[u_i,\mathfrak{m}]_\mathfrak{m}&=&
[u_i,\hat{\mathfrak{m}}_0]_\mathfrak{m}
+[u_i,\sum_{\alpha'\in\mathfrak{t}\cap\mathfrak{h},\alpha'\neq0}
\hat{\mathfrak{m}}_{\pm\alpha'}]_\mathfrak{m}\\
&\subset&[u_i,\hat{\mathfrak{m}}_0]_\mathfrak{m}
+\sum_{\alpha'\in\mathfrak{t}\cap\mathfrak{h},\alpha'\neq0}
\hat{\mathfrak{m}}_{\pm\alpha'}.
\end{eqnarray*}

By Lemma \ref{ortho-lemma-1} and (\ref{005}),
we have
\begin{eqnarray*}
\langle u_i,[u_i,\mathfrak{m}]_\mathfrak{m}\rangle^F_{u_i}
\subset\langle u_i,[u_i,\hat{\mathfrak{m}}_0]\rangle^F_{u_i}
+\langle u_i,
\sum_{\alpha'\in\mathfrak{t}\cap\mathfrak{h},\alpha'\neq0}
\hat{\mathfrak{m}}_{\pm\alpha'}\rangle^F_{u_i}
=0.
\end{eqnarray*}

On the other hand, $u_1$ and $u_2$ are nonzero vectors in a compact subalgebra of type $\mathrm{A}_1$, so they generate $\mathrm{S}^1$-subgroups. By Lemma \ref{lemma-6}, we have $u_1$ and
$u_2$ are contained in $\mathcal{C}$.

This proves the lemma when $\dim\hat{\mathfrak{m}}_0=3$.
\end{proof}

\begin{lemma}\label{key-lemma-2}
Assume $(G/H,F)$ is an odd dimensional compact connected homogeneous Finsler space
satisfying Assumption (I)
and keep all relevant notations and assumptions. Let $\alpha$ be
a root of $\mathfrak{g}$ such that
$\alpha\in\mathfrak{t}\cap\mathfrak{h}$,
 and $\alpha$ is the only root of $\mathfrak{g}$ contained in $\mathbb{R}_{>0}\alpha+\mathfrak{t}\cap\mathfrak{m}$.
Then $\alpha$ is a root of $\mathfrak{h}$ and
$\mathfrak{h}_{\pm\alpha}=\mathfrak{g}_{\pm\alpha}
=\hat{\mathfrak{g}}_{\pm\alpha}$.
\end{lemma}

\begin{proof}
We assume conversely that $\alpha$ is not a root of $\mathfrak{h}$, then $\mathfrak{g}_{\pm\alpha}\subset\mathfrak{m}$.
Denote $\mathfrak{t}'$ the bi-invariant orthogonal complement
of $\alpha$ in $\mathfrak{t}\cap\mathfrak{h}$, and $T'$ the
torus in $H$ generated by $\mathfrak{t}'$.
Then $\mathrm{Fix}_o(T',M)$ is totally geodesic in $(M,F)$.

Assume $\dim\hat{\mathfrak{m}}_0=1$, then by the fixed point set technique, $\mathrm{Fix}_o(T',M)$ is a compact coset space $G'/H'$
such that $\mathfrak{g}'=\mathrm{Lie}(G')=\mathbb{R}\oplus \mathrm{A}_1$ and $\mathfrak{h}'=\mathbb{R}$ is contained in the $\mathrm{A}_1$-factor of
$\mathfrak{g}'$. So $(\mathrm{Fix}_o(T',M),F|_{\mathrm{Fix}_o(T',M)})$
is finitely covered by and locally isometric to the homogeneous Finsler space $(M',F')$ where $M'=\mathrm{S}^1\times \mathrm{S}^2=\mathrm{SO}(2)\times \mathrm{SO}(3)/ e\times \mathrm{SO}(2)$ and $F'$ is $\mathrm{SO}(2)\times \mathrm{SO}(3)$-invariant. By Example 2 in Section \ref{Section-homogeneous-examples} and Proposition \ref{prop-finite-cover}, we can find a sequence of prime closed geodesics
$\gamma_n$ for $(\mathrm{Fix}_o(T',M),F|_{\mathrm{Fix}_o(T',M)})$
as well as for $(M,F)$, such that their lengths $l(\gamma_n)$ diverge to infinity. They can not belong to the same orbit, which
contradicts Assumption (I).

Assume $\dim\hat{\mathfrak{m}}_0=3$, then $(\mathrm{Fix}_o(T',M),F|_{\mathrm{Fix}_o(T',M)}$
is finitely covered by and locally isometric to $M'=\mathrm{S}^3\times \mathrm{S}^2=\mathrm{SU}(2)\times \mathrm{SO}(3)/e\times \mathrm{SO}(2)$ with an $\mathrm{SU}(2)\times \mathrm{SO}(3)$-invariant metric. Using Example 5 in Section \ref{Section-homogeneous-examples} instead, we can apply similar
argument as the previous case to get a contradiction to Assumption (I).

This ends the proof of this lemma.
\end{proof}

\begin{lemma}\label{key-lemma-3}
Assume $(G/H,F)$ is an odd dimensional compact connected reversible
homogeneous Finsler space
and keep all relevant notations and assumptions. Let $\alpha$ be
a root of $\mathfrak{g}$ such that
$\mathfrak{g}_{\pm\alpha}\subset\mathfrak{m}$,
$\mathrm{pr}_{\mathfrak{h}}\alpha\neq0$, and
$\alpha$ is the only root of $\mathfrak{g}$ contained
in $(2\mathbb{N}-1)\alpha+\mathfrak{t}\cap\mathfrak{m}$.
Then $\mathfrak{g}_{\pm\alpha}\cap\mathcal{C}\neq\emptyset$.
\end{lemma}

\begin{proof}
Denote $\mathfrak{t}'$ the bi-invariant orthogonal complement of
$\alpha'=\mathrm{pr}_{\mathfrak{h}}(\alpha)$ in $\mathfrak{t}\cap\mathfrak{h}$. With respect to $\mathfrak{t}'$,
we have the decomposition
$\mathfrak{m}=\sum_{\beta''\in\mathfrak{t}'}
\hat{\hat{\mathfrak{m}}}_{\pm\beta''}$ (see the detailed description after (\ref{T'-decomp-m})).

Let $u$ be any vector in $\mathfrak{g}_{\pm\alpha}$ with $|u|_\mathrm{bi}=1$, then we also have $u\in\hat{\hat{\mathfrak{m}}}_0$.

Direct calculation shows
\begin{equation}\label{009}
[u,\mathfrak{m}]_\mathfrak{m}\subset
\hat{\mathfrak{m}}_0+[\mathfrak{t}\cap\mathfrak{m},u]+\sum_{\beta''\in\mathfrak{t}',
\mathrm{pr}_{\mathfrak{t}'}(\beta'')\neq0}
\hat{\hat{\mathfrak{m}}}_{\pm\beta''}.
\end{equation}
By Lemma \ref{ortho-lemma-2}, Lemma \ref{ortho-lemma-2-5} and
Lemma \ref{ortho-lemma-2-6}, each summation factor in the right
side of (\ref{009}) is $g_u^F$-orthogonal to $u$.
So we have $\langle u,[u,\mathfrak{m}]_\mathfrak{m}\rangle_u^F=0$.
Because $u$ is contained in a subalgebra of type $\mathrm{A}_1$, $u$ generates an $\mathrm{S}^1$-subgroup.

This proves $u\in\mathcal{C}$ and ends the proof of the lemma.
\end{proof}

When Lemma \ref{key-lemma-3} is applied to find obstacle to Assumption (I), it is often accompanied with the following lemma.

\begin{lemma}\label{key-lemma-4}
Assume $(G/H,F)$ is an odd dimensional compact connected homogeneous Finsler space
such that $G=I_0(M,F)$ is a compact simple Lie group and $(M,F)$
satisfies Assumption (I). Then we have the following:
\begin{description}
\item{\rm (1)} If $\Delta\cap\mathfrak{m}=\emptyset$,
i.e. there exist no roots of $\mathfrak{g}$ contained in $\mathfrak{t}\cap\mathfrak{m}$, then for any root plane $\mathfrak{g}_{\pm\beta}$ of $\mathfrak{g}$, $\mathcal{C}\cap\mathfrak{g}_{\pm\beta}=\emptyset$.
\item{\rm (2)} If $\mathfrak{t}\cap\mathfrak{m}$ contains a root
$\alpha$ of $\mathfrak{g}$, then for any root $\beta$ of $\mathfrak{g}$ with
$|\alpha|_{\mathrm{bi}}\neq|\beta|_{\mathrm{bi}}$, we have
$\mathcal{C}\cap\mathfrak{g}_{\pm\alpha}=\emptyset$.
\end{description}
\end{lemma}

\begin{proof}
(1) Assume conversely that $\mathcal{C}$ contains some vector
$v\in\mathfrak{g}_{\pm\beta}$.
For simplicity, we choose the bi-invariant inner product
on $\mathfrak{g}$ such that $|\beta|_\mathrm{bi}=1$. By Lemma \ref{key-lemma-1}, there exists a vector $u\in\mathfrak{t}\cap\mathfrak{m}$ which is contained in $\mathcal{C}$.
Because $v$ is $\mathrm{Ad}(G)$-conjugation to $\beta$, we get
$\mathcal{C}=\mathrm{Ad}(H)u\subset
\mathrm{Ad}(G)\beta$,
i.e. the two vectors $u$ and $\beta$ in $\mathfrak{t}$
are in the same $\mathrm{Ad}(G)$-orbit. By Proposition 2.2, Chapter 7
in \cite{Hegalson}, $u$ and $\beta$ belong to the same Weyl group
orbit. This is impossible because the Weyl group orbit of $\beta$ consists of all roots of $\mathfrak{g}$ with the same length,
but $u$ is not a root of $\mathfrak{g}$.
This proves the statement (1).

(2) Assume conversely that $\mathcal{C}$ contains some vector
$v\in\mathfrak{g}_{\pm\beta}$ such that
$|\alpha|_\mathrm{bi}\neq|\beta|_{\mathrm{bi}}$. For simplicity,
we choose the bi-invariant inner product on $\mathfrak{g}$ such
that $|\beta|_\mathrm{bi}=1$ and $|\alpha|_\mathrm{bi}=c\neq1$.
By Lemma \ref{key-lemma-1}, $\hat{\mathfrak{m}}_0$
contains a vector $u\in\mathcal{C}$, which is $\mathrm{Ad}(G)$-conjugation to $c^{-1}\alpha$. Meanwhile $v$
is $\mathrm{Ad}(G)$-conjugate to $\alpha$. So $c^{-1}\alpha$ and $\beta$ belong to the same $\mathrm{Ad}(G)$-orbit because
$\mathcal{C}\subset
\mathrm{Ad}(G)\cdot(c^{-1}\alpha)\cap\mathrm{Ad}(G)\beta
$. By Proposition 2.2, Chapter 7 in \cite{Hegalson}, $c^{-1}\alpha$ and $\beta$ belong to the same Weyl group orbit.
This is impossible because $c^{-1}\alpha$ is not a root of
$\mathfrak{g}$.
\end{proof}

Now we start the proof of Theorem \ref{classification-thm-2}.
We follow the theme in \cite{XD-towards-odd}, i.e. we divide the discussion into
three cases.
\begin{description}
\item{\bf Case I.} Each root plane of $\mathfrak{h}$ is a root plane of $\mathfrak{g}$.
\item{\bf Case II.} There exists a root plane $\mathfrak{h}_{\pm\alpha'}$ of $\mathfrak{h}$ which is not a root plane of $\mathfrak{g}$, and there are two roots $\alpha$ and $\beta$ of $\mathfrak{g}$ from different simple ideals, such that
    $\mathrm{pr}_\mathfrak{h}(\alpha)=
    \mathrm{pr}_\mathfrak{h}(\beta)$.
\item{\bf Case III.} The same as Case II except that the different roots $\alpha$ and $\beta$ are from the same simple ideal of $\mathfrak{g}$.
\end{description}

In the rest of this section, we will discuss Case I and Case II, and leave Case III to the next section.

{\bf Proof of Theorem \ref{classification-thm-1} in Case I.}
Assume $(M,F)=(G/H,F)$ is an odd dimensional compact connected homogeneous Finsler space belonging to Case I and
satisfying Assumption (I).

The deck transformation induces an injective group endomorphism
from $H/H_0$ to $\pi_1(M)$. By Assumption (I), $\pi_1(M)$ must
be a cyclic group, and so does $H/H_0$. Denote $g\in H$ a representative for a generator of $H/H_0$. Because $\mathrm{Ad}(g)H=H$, we also have
$\mathrm{Ad}(g)\hat{\mathfrak{m}}_0=\hat{\mathfrak{m}}_0$.

By Lemma \ref{key-lemma-1}, there exists a vector
$u_1\in\hat{\mathfrak{m}}_0\cap\mathcal{C}$.
Because in Case I, $[\mathfrak{h},\hat{\mathfrak{m}}_0]=0$, it
is easy to get
$$\mathcal{C}=\{\mathrm{Ad}(g^k)u_1\mbox{ with all }k\in\mathbb{Z}\}\subset\hat{\mathfrak{m}}_0.$$
By Lemma \ref{key-lemma-1}, $\mathcal{C}$ contains at least two points.

We claim $\mathrm{Ad}(g)$ fixes some nonzero vector $u\in\mathfrak{m}$, which can be proved as following.

Let $\mathfrak{t}''$ be a Cartan subalgebra of
$\mathfrak{g}$ such that $T''=\exp\mathfrak{t}''$ contains $g$.
Then $\mathrm{Ad}(g)$ fixes each vector in $\mathfrak{t}''$.
Obviously $\mathfrak{t}''$ is not contained in $\mathfrak{h}$,
i.e. $\mathrm{pr}_\mathfrak{m}(\mathfrak{t}'')\neq0$.

Because $H_0$ is normal in $H$, we have $\mathrm{Ad}(g)\mathfrak{h}=\mathfrak{h}$, and $\mathrm{Ad}(g)$
preserves
the bi-invariant orthogonal reductive decomposition $\mathfrak{g}=\mathfrak{h}+\mathfrak{m}$.
So $\mathrm{Ad}(g)$ fixes each
vector in $\mathrm{pr}_\mathfrak{m}(\mathfrak{t}'')$.
Any nonzero vector $u\in\mathrm{pr}_\mathfrak{m}(\mathfrak{t}'')$
meets the requirement.

This ends the proof of the claim.

The previous claim implies
$M'=\mathrm{Fix}_o(g,M)$ has a positive dimension.
By Lemma \ref{key-lemma-1}, it is obvious that
$T_oM'\cap\mathcal{C}=\emptyset$. The submanifold
$(M',F|_{M'})$ is a homogeneous Finsler space. By Lemma \ref{lemma-1} and the homogeneity, there must exist
one closed geodesic $c(t)$ on $(M',F')$ such that $c(0)=o$, $\dot{c}(0)=u'\in\mathfrak{m}$ and $|u'|_\mathrm{bi}=1$.
Because $M'$ is totally geodesic in $M$, we have $u'\in\mathcal{C}\cap T_oM'$, which contradicts our previous
observation.

This proves Theorem \ref{classification-thm-2} in Case I, which
can be summarized as

\begin{proposition}\label{prop-Case-I}
If $(M,F)=(G/H,F)$ with $G=I_0(M,F)$ is an odd dimensional compact connected homogeneous Finsler space in Case I,  then it has at least two $\hat{G}$-orbits of prime closed geodeics.
\end{proposition}

{\bf Proof of Theorem \ref{classification-thm-2} in Case II.}
Assume  $(M,F)=(G/H,F)$ is an odd dimensional compact connected reversible homogeneous Finsler space belonging to Case II and
satisfying Assumption (I).

We have a Lie algebra direct sum
decomposition $\mathfrak{g}=\mathfrak{g}_1\oplus\mathfrak{g}_2\oplus\mathfrak{g}_3$
in which $\mathfrak{g}_1$ and $\mathfrak{g}_2$ are simple, such that
there are roots of $\mathfrak{g}$, $\alpha\in\mathfrak{t}\cap\mathfrak{g}_1$ and
$\beta\in\mathfrak{t}\cap\mathfrak{g}_2$,
and $\alpha'=\mathrm{pr}_\mathfrak{h}(\alpha)
=\mathrm{pr}_\mathfrak{h}(\beta)$ is a root of $\mathfrak{h}$.
Obviously we have
\begin{eqnarray*}
\hat{\mathfrak{m}}_0=\mathfrak{t}\cap\mathfrak{m}=
\mathbb{R}(\alpha-\beta),\quad\mbox{and}\quad
\mathfrak{h}_{\pm\alpha'}\subset
\hat{\mathfrak{m}}_{\pm\alpha'}=\mathfrak{g}_{\pm\alpha}
+\mathfrak{g}_{\pm\beta}.
\end{eqnarray*}

By Lemma \ref{key-lemma-1}, we can find a vector $u\in\mathcal{C}\cap\hat{\mathfrak{m}}_0=\mathcal{C}\cap\mathfrak{t}$.

If there exists a root $\delta$ of $\mathfrak{g}_1$, such that
$\delta\neq \pm\alpha$ and $\delta$ is not bi-invariant orthogonal
to $\alpha$, then $\mathfrak{g}_{\pm\delta}\subset\mathfrak{m}$,
and by Lemma \ref{key-lemma-3}, there exists a vector $v\in\mathfrak{g}_{\pm\delta}\cap\mathcal{C}$. By Assumption (I),
$\mathcal{C}=\mathrm{Ad}(H)v
\subset\mathrm{Ad}(G)v\subset\mathfrak{g}_1$.
But $u\in\mathcal{C}$ is not contained in $\mathfrak{g}_1$.
This is a contradiction. To summarize, $\mathfrak{g}_1=\mathrm{A}_1$, otherwise it can not be simple. Similarly, we also have $\mathfrak{g}_2=\mathrm{A}_1$.

Next, we prove $\mathfrak{g}_3$ must be zero. Assume conversely it is not, by Lemma \ref{lemma-no-abel-factor}, $\mathfrak{g}_3$ is
semi-simple. Because $\mathfrak{t}\cap\mathfrak{m}\subset
\mathfrak{t}\cap(\mathfrak{g}_1+\mathfrak{g}_2)$, we have $\mathfrak{t}\cap\mathfrak{g}_3\subset\mathfrak{h}$. By Lemma \ref{key-lemma-2}, any root plane in $\mathfrak{g}_3$ is also contained in $\mathfrak{h}$. So the ideal $\mathfrak{g}_3$ of $\mathfrak{g}$ is contained in $\mathfrak{h}$. But $G=I_0(M,F)$ acts transitively on $M=G/H$. This is a contradiction.

So in this case $\mathfrak{h}$ is a diagonal $\mathrm{A}_1$
in $\mathfrak{g}=\mathrm{A}_1\oplus \mathrm{A}_1$, so $(G/H,F)$ is locally isometric to a Riemannian symmetric $\mathrm{S}^3$. By Proposition \ref{prop-classification-local}, $(M,F)$ is either the Riemannian
symmetric $\mathrm{S}^3$ or the Riemanninan symmetric $\mathbb{R}\mathrm{P}^3$.

This proves Theorem \ref{classification-thm-2} in Case II, which
can be summarized as following.

\begin{proposition}\label{prop-Case-II}
If $(M,F)=(G/H,F)$ with $G=I_0(M,F)$ is an odd dimensional  compact connected reversible homogeneous Finsler space in Case II
satisfying Assumption (I), then $(M,F)$ is either the Riemannian
symmetric $\mathrm{S}^3$ or the Riemannian symmetric $\mathbb{R}\mathrm{P}^3$.
\end{proposition}

\section{Proof of Theorem \ref{classification-thm-2} in Case III}

We continue the proof of Theorem \ref{classification-thm-2}
in Case III.
Let $(M,F)=(G/H,F)$ with $G=I_0(M,F)$ be
an odd dimensional compact connected homogeneous Finsler space satisfying Assumption (I). We keep all relevant notations and assumptions.
In particular, there exists a root plane $\mathfrak{g}_{\pm\alpha'}$ of $\mathfrak{h}$, which is not a root plane of $\mathfrak{g}$, and there exists a pair of different roots
$\alpha$ and $\beta$ from the same simple ideal in $\mathfrak{g}$,
such that $\mathrm{pr}_\mathfrak{h}(\alpha)=\mathrm{pr}_\mathfrak{h}(\beta)
=\alpha'$.

Recall that $\mathcal{C}\subset\mathfrak{m}$ is the subset of all vectors $u$'s such that $|u|_\mathrm{bi}=1$ and the geodesic of
$(M,F)$ passing $o$ in the direction of $u$ is closed. By Assumption (I), $\mathcal{C}$ is an $\mathrm{Ad}(H)$-orbit.

Using similar technique as in the last section, it is not hard to see that $G$ must be simple, i.e. we have the following lemma.

\begin{lemma}\label{lemma-odd-simple-g}
For any odd dimensional compact connected homogeneous Finsler space $(M,F)=(G/H,F)$
in Case III, with $G=I_0(M,F)$ and satisfying Assumption (I), the group $G$ is simple.
\end{lemma}

\begin{proof}
Assume conversely that $G$ is not simple. By Lemma \ref{lemma-no-abel-factor}, $G$ is semi-simple. So we
have a nontrivial direct sum decomposition
$\mathfrak{g}=\mathfrak{g}_1\oplus\mathfrak{g}_2$,
in which $\mathfrak{g}_1$ is a simple ideal, from which
we get the roots $\alpha$ and $\beta$ indicated in Case III,
and $\mathfrak{g}_2$ is a semi-simple ideal.
Obviously $\mathfrak{t}\cap\mathfrak{m}=\mathbb{R}(\alpha-\beta)
\in\mathfrak{g}_1$, so $\mathfrak{t}\cap\mathfrak{g}_2\in\mathfrak{h}$. By Lemma \ref{key-lemma-2},
any root plane of $\mathfrak{g}_2$ is contained in $\mathfrak{h}$.
So we have $\mathfrak{g}_2\subset\mathfrak{h}$, which contradicts
to the fact that $G=I_0(M,F)$ acts transitively on $M=G/H$.
This ends the proof of this lemma.
\end{proof}

{\bf Proof of Theorem \ref{classification-thm-2} in Case III.} We will check
case by case each possible type of $\mathfrak{g}=\mathrm{Lie}(G)$ from $\mathrm{A}_n$ to $\mathrm{G}_2$, and
each possible unordered pair of roots $\alpha$ and $\beta$ of $\mathfrak{g}$.

We follow the convention in \cite{XD-normal-homogeneous}\cite{XD-towards-odd}\cite{XW-KVCL-Riemannian-normal}
for presenting roots
of the compact simple Lie algebra $\mathfrak{g}$.
We choose a suitable bi-invariant inner product on $\mathfrak{g}$, and isometrically identify
$\mathfrak{t}$ with an Euclidean space $\mathbf{V}$ satisfying:
\begin{description}
\item{\rm (1)} When $\mathfrak{g}=\mathrm{A}_n$, we denote
$\{e_1,\ldots,e_{n+1}\}$ the standard orthonormal basis of $\mathbb{R}^{n+1}$, and $\mathbf{V}$
the orthogonal complement of $e_1+\cdots+e_{n+1}$;
\item{\rm (2)} When $\mathrm{rk}\mathfrak{g}=n$ and $\mathfrak{g}\neq \mathrm{A}_n$, we denote $\{e_1,\ldots,e_n\}$
the standard orthonormal basis of $\mathbf{V}=\mathbb{R}^n$;
\item{\rm (3)} Under this identification, the root system $\Delta$ of $\mathfrak{g}$ is presented as in Table \ref{table-1}.
\end{description}

\begin{table}
  \centering
\begin{tabular}{|c|c|}
  \hline
  $\mathfrak{g}$ & Root system $Delta$ of $\mathfrak{g}$ \\ \hline
  $\mathrm{A}_n$ & $\pm(e_i-e_j)$, $\forall 1\leq i<j\leq n+1$ \\ \hline
  $\mathrm{B}_n$ & $\pm e_i\pm e_j$, $\forall 1\leq i<j\leq n$; $\pm e_i$, $\forall 1\leq i\leq n$ \\ \hline
  $\mathrm{C}_n$ & $\pm e_i\pm e_j$, $\forall 1\leq i<j\leq n$; $\pm 2e_i$, $\forall 1\leq i\leq n$ \\ \hline
  $\mathrm{D}_n$ & $\pm e_i\pm e_j$, $\forall 1\leq i<j\leq n$ \\ \hline
  $\mathrm{E}_6$ & $\pm e_i\pm e_j$, $\forall 1\leq i<j<6$; \\
        & $\frac12(\pm e_1\pm\cdots\pm e_5\pm\sqrt{3}e_6)$ with odd plus signs \\ \hline
  $\mathrm{E}_7$ & $\pm e_i\pm e_j$, $\forall 1\leq i<j<7$; $\pm\sqrt{2}e_7$;\\
        & $\pm\frac12 e_1\pm\cdots\pm\frac12 e_6\pm\frac{\sqrt{2}}2 e_7$ with even $+\frac12$'s \\ \hline
  $\mathrm{E}_8$ & $\pm e_i\pm e_j$, $\forall 1\leq i<j\leq 8$; \\
        & $\frac12(\pm e_1\pm\cdots\pm e_8)$ with even plus signs \\ \hline
  $\mathrm{F}_4$ & $\pm e_i\pm e_j$, $\forall 1\leq i<j\leq 4$; $\pm e_i$, $\forall 1\leq i\leq 4$;\\
        & $\frac12(\pm e_1\pm e_2\pm e_3\pm e_4)$\\ \hline
  $\mathrm{G}_2$ & $\pm e_1,\pm\frac12e_1\pm\frac{\sqrt{3}}2e_2,
\pm\sqrt{3}e_2,
\pm\frac32e_1\pm\frac{\sqrt{3}}2e_2$\\
  \hline
\end{tabular}
  \caption{Root system of compact simple Lie algebras}\label{table-1}
\end{table}

Using the Weyl group action on the pair $\alpha$ and $\beta$, we can reduce the case number significantly. Moreover, when $\mathfrak{g}=\mathrm{D}_4$ or $\mathrm{E}_6$, we can use outer automorphism to change the pair $\alpha=e_1+e_2$ and $\beta=-e_3-e_4$ to $\alpha=e_1+e_2$ and $\beta=e_2-e_1$.
But still, many subcases remain.
Observing that for many subcases we can apply similar argument, we sort all the subcases into
five groups, from Case III-A to Case III-E.

{\bf Case III-A.}
Assume for the root $\alpha'$ of $\mathfrak{h}$ described in Case III, we can find two roots
$\alpha$ and $\beta$ of $\mathfrak{g}$ such that
$\mathrm{pr}_\mathfrak{h}(\alpha)=\mathrm{pr}_\mathfrak{h}(\beta)=\alpha'$
and the angle
$\theta_{\alpha,\beta}$ between $\alpha$ and $\beta$ is
$\pi/3$ and $2\pi/3$.

The angles are defined with respect to the chosen bi-invariant inner product on $\mathfrak{g}$.
Notice that all possible angles between
the two roots $\alpha$ and $\beta$ in Case III are $\pi/6$,
$\pi/4$, $\pi/3$, $\pi/2$, $2\pi/3$, $3\pi/4$ and $5\pi/6$.

The following lemma provides the contradiction.

\begin{lemma}\label{lemma-no-pi-over-3-or-2-pi-over-3}
Assume $(M,F)=(G/H,F)$ is an odd dimensional compact connected reversible homogeneous Finsler space in Case III such that $G=I_0(M,F)$ and $(M,F)$ satisfies Assumption (I). Then for the roots $\alpha$ and
$\beta$ in Case III, their angle $\theta_{\alpha,\beta}$ can not be $\pi/3$ or $2\pi/3$.
\end{lemma}

\begin{proof}
Firstly, we assume that $\mathfrak{g}\neq \mathrm{G}_2$ and the angle between
$\alpha$ and $\beta$ is $\pi/3$. Then
$$\mathfrak{g}'=\mathbb{R}\alpha+\mathbb{R}\beta+
\mathfrak{g}_{\pm\alpha}+\mathfrak{g}_{\pm\beta}+
\mathfrak{g}_{\pm(\alpha-\beta)}$$
is a subalgebra of type $\mathrm{A}_2$ in $\mathfrak{g}$,
and $\mathfrak{g}'\cap\mathfrak{h}=\mathbb{R}\alpha'+
\mathfrak{h}_{\pm\alpha'}$ is a subalgebra of type $\mathrm{A}_1$.
As in the proof of Lemma 18 in \cite{XD-normal-homogeneous}, direct calculation for matrices in $su(3)$ shows this pair $(\mathfrak{g}',\mathfrak{h}')$ can not exist. This is the contradiction.

Secondly, we assume that $\mathfrak{g}\neq \mathrm{G}_2$ and the angle between
$\alpha$ and $\beta$ is $2\pi/3$. In this case, $\mathfrak{t}\cap\mathfrak{m}=\mathbb{R}(\alpha-\beta)$
does not contain roots of $\mathfrak{g}$.
On the other hand, $2\alpha'=\alpha+\beta$
is a root of $\mathfrak{g}$, but not a root of $\mathfrak{h}$ because $\alpha'$ is. So we have
$\mathfrak{g}_{\pm(\alpha+\beta)}=\hat{\mathfrak{m}}_{\pm2\alpha'}
\subset\mathfrak{m}$. By Lemma \ref{key-lemma-3},
$\mathcal{C}\cap\mathfrak{g}_{\pm(\alpha+\beta)}\neq \emptyset$.
This is a contradiction to (1) of Lemma \ref{key-lemma-4}.

Finally, we assume that $\mathfrak{g}=\mathrm{G}_2$,
and the angle between $\alpha$ and $\beta$ is $\pi/3$ and $2\pi/3$
respectively. We can apply (2) of Lemma \ref{key-lemma-4} and similar argument in the previous paragraph.

This proves the lemma for all possible cases.
\end{proof}

To summarize, any odd dimensional compact connected reversible homogeneous Finsler space $(M,F)$ in Case III-A  can not satisfy Assumption (I). In the discussion below, we only need to consider other angles. Notice that
Case III-A also covers the subcases $\theta_{\alpha,\beta}=\pi/6$ and $\pi/2$ when $\mathfrak{g}=\mathrm{G}_2$.

{\bf Case III-B.} Assume $F$ is reversible. Then Each row in Table \ref{table-2} provides
a subcase for which we can find a contradiction as following.

We can find the root $\delta$ of $\mathfrak{g}$, as listed in Table \ref{table-2}, such that
$\mathfrak{g}_{\pm\delta}\cap\mathcal{C}\neq\emptyset$ by Lemma \ref{key-lemma-3}.
Notice that
$\mathfrak{t}\cap\mathfrak{m}=\mathbb{R}(\alpha-\beta)$ does not contains any root of $\mathfrak{g}$.
By (1) of Lemma \ref{key-lemma-4}, $\mathfrak{g}_{\pm\delta}\cap\mathcal{C}=\emptyset$. This is a contradiction.

\begin{table}
  \centering
  \begin{tabular}{|c|c|c|c|c|}
     \hline
   No.   & $\mathfrak{g}$ & $\alpha$ & $\beta$   & $\delta$  \\ \hline
   1   & $\mathrm{A}_n$, $n>3$ & $e_1-e_3$ & $e_4-e_2$  & $e_1-e_5$ \\ \hline
   2   & $\mathrm{B}_n$, $n>3$ & $e_1+e_2$ & $-e_3-e_4$ & $e_1$     \\ \hline
   3   & $\mathrm{B}_n$, $n>3$ & $e_1+e_2$ & $-e_3$ & $e_3+e_4$\\ \hline
   4   & $\mathrm{B}_n$, $n>1$ & $e_1+e_2$ & $-e_1$     & $e_2-e_1$ \\ \hline
   5   & $\mathrm{C}_n$, $n>2$ & $2e_1$    & $-e_2-e_3$ & $2e_3$    \\ \hline
   6   & $\mathrm{C}_n$, $n>3$ & $e_1+e_2$ & $-e_3-e_4$ & $2e_4$    \\ \hline
   7   & $\mathrm{C}_n$, $n>2$ & $2e_1$    & $-e_1-e_2$ & $2e_3$    \\ \hline
   8   & $\mathrm{D}_n$, $n>4$ & $e_1+e_2$ & $-e_3-e_4$ & $e_4+e_5$ \\ \hline
   9   & $\mathrm{E}_6$        & $e_1+e_2$ & $e_2-e_1$  & $\frac12(e_1+\cdots+e_5-\sqrt{3}e_6)$ \\ \hline
   10   & $\mathrm{E}_7$        & $e_1+e_2$ & $e_2-e_1$  & $\frac12(e_1+\cdots+e_6+\sqrt{2}e_7)$ \\ \hline
   11  & $\mathrm{E}_8$        & $e_1+e_2$ & $e_2-e_1$  & $\frac12(e_1+\cdots+e_8)$ \\ \hline
   12  & $\mathrm{E}_8$        & $e_1+e_2$ & $-e_3-e_4$ & $\frac12(-e_1-e_2+e_3+\cdots+e_8)$ \\  \hline
   13  & $\mathrm{F}_4$        & $e_1+e_2$ & $-e_3$     & $\frac12(e_1-e_2-e_3+e_4)$ \\ \hline
   14  & $\mathrm{F}_4$        & $e_1+e_2$ & $-e_2$     & $\frac12(e_1+e_2+e_3+e_4)$ \\ \hline
   15  & $\mathrm{G}_2$        & $-\frac32e_1+\frac{\sqrt{3}}2e_2$ & $e_1$ & $\sqrt{3}e_2$ \\ \hline
   \end{tabular}
  \caption{Subcases in Case III-B}\label{table-2}
\end{table}

To summarize, any compact connected reversible homogeneous Finsler space $(M,F)$ in Case III-B  can not satisfy Assumption (I).

{\bf Case III-C.} Each row of Table \ref{table-3} provides a subcase for which we can find the contradiction as following.

\begin{table}
  \centering
  \begin{tabular}{|c|c|c|c|c|}
    \hline
    No.  & $\mathfrak{g}$ & $\alpha$ & $\beta$ & Roots of $\mathfrak{g}'$ \\ \hline
    1   & $\mathrm{B}_n$, $n>1$   & $e_1$    & $e_2$   & $\pm e_1$, $\pm e_2$, $\pm e_1\pm e_2$ \\ \hline
    2   & $\mathrm{C}_n$, $n>2$   & $e_1+e_2$& $e_2-e_1$ & $\pm e_1\pm e_2$, $\pm 2e_1$, $\pm 2e_2$ \\ \hline
    3   & $\mathrm{F}_4$          & $e_1$    & $e_2$     & $\pm e_1$, $\pm e_2$, $\pm e_1\pm e_2$ \\ \hline
  \end{tabular}
  \caption{Subcases in Case III-C}\label{table-3}
\end{table}

In the Lie algebra $\mathfrak{g}$, we can find a regular subalgebra $\mathfrak{g}'=\mathrm{C}_2$, which roots are listed in
Table \ref{table-3}, such that $\mathfrak{g}'\cap\mathfrak{h}=\mathrm{A}_1$, $\alpha'=\mathrm{pr}_\mathfrak{h}(\alpha)$ is half of a long root of $\mathfrak{g}'$, and
$\mathfrak{h}_{\pm\alpha'}\subset\mathfrak{g}'\cap\mathfrak{h}$ is contained in the sum of the
two root planes for short roots of $\mathfrak{g}'$. We can identify $\mathfrak{g}'$ with
the matrix Lie algebra $\mathrm{sp}(2)$, such that $\mathfrak{g}'\cap\mathfrak{h}$ is linearly spanned by
\begin{equation}\label{008}
H=\left(
      \begin{array}{cc}
        \mathbf{i} & 0 \\
        0 & 0 \\
      \end{array}
    \right),\quad
X=\left(
    \begin{array}{cc}
      0 & a \\
      -\bar{a} & 0 \\
    \end{array}
  \right),\quad\mbox{and}\quad
Y=[H,X]=\left(
          \begin{array}{cc}
            0 & \mathbf{i}a \\
            \bar{a}\mathbf{i} & 0 \\
          \end{array}
        \right),
\end{equation}
where $a\in\mathbb{H}$ is nonzero.
But direct calculation shows,
$$[X,Y]=\left(
          \begin{array}{cc}
            2|a|^2\mathbf{i} & 0 \\
            0 & -2\bar{a}\mathbf{i}a \\
          \end{array}
        \right)\notin\mathfrak{g}'\cap\mathfrak{h}.
$$
This is a contradiction.

To summarize, Case III-C can not happen.

{\bf Case III-D.} Each row of Table \ref{table-4} provides a subcase for which we can prove $(M,F)$
is locally isometric to a Riemannnian symmetric sphere.

\begin{table}
  \centering
  \begin{tabular}{|c|c|c|c|c|c|}
     \hline
     No.  & $\mathfrak{g}$ & $\alpha$ & $\beta$ & $\mathfrak{h}$ & Universal cover of $G/H$ \\ \hline
     1 & $\mathrm{B}_3$ & $e_1+e_2$ & $-e_3$ & $\mathrm{G}_2$ & $\mathrm{S}^7=\mathrm{Spin}(7)/\mathrm{G}_2$ \\ \hline
     2 & $\mathrm{D}_n$, $n>2$ & $e_1+e_2$ & $e_2-e_1$ & $\mathrm{B}_{n-1}$ & $\mathrm{S}^{2n-1}=\mathrm{SO}(2n)/\mathrm{SO}(2n-1)$ \\ \hline
   \end{tabular}
  \caption{Subcases in Case III-D}\label{table-4}
\end{table}

We take No.  2 in Table \ref{table-4} as the example. The argument for the other is similar.

By Lemma \ref{key-lemma-2}, $\pm e_i\pm e_j$ with $1\leq i<j\leq n$ are roots of $\mathfrak{h}$,
and $\mathfrak{h}_{\pm(e_i\pm e_j)}=\mathfrak{g}_{\pm (e_i\pm e_j)}$.
Take any nonzero $u\in\mathfrak{g}_{\pm (e_2-e_i)}$ for $i>2$, $\mathrm{ad}(u)=[u,\cdot]$ defines
a linear isomorphism
$$\mathrm{ad}(u):\hat{\mathfrak{g}}_{\pm e_2}=\mathfrak{g}_{\pm(e_1+e_2)}+\mathfrak{g}_{\pm(e_2-e_1)}
\rightarrow\mathfrak{g}_{\pm(e_1+e_i)}+\mathfrak{g}_{\pm(e_i-e_1)}=
\hat{\mathfrak{g}}_{\pm e_i},
$$
and because $u\in\mathfrak{h}$, $\mathrm{ad}(u)$ maps $\hat{\mathfrak{g}}_{\pm e_2}\cap\mathfrak{h}$ and $\hat{\mathfrak{g}}_{\pm e_2}\cap\mathfrak{m}$ to $\hat{\mathfrak{g}}_{\pm e_i}\cap\mathfrak{h}$ and $\hat{\mathfrak{g}}_{\pm e_i}\cap\mathfrak{m}$ respectively. This implies that $\pm e_i$'s for all $i>1$
are roots of $\mathfrak{h}$.

To summarize, we have found all the roots of $\mathfrak{h}$, i.e.
$\pm e_i$ for all $i>1$ and $\pm e_i\pm e_j$ for all $1<i<j\leq n$.
So $\mathfrak{h}=\mathrm{B}_{n-1}$. The argument in Subcase 1, Subection 6.3 in \cite{XD-normal-homogeneous}
shows this $\mathfrak{h}$ is unique up to an $\mathrm{Ad}(G)$-action. So
$(G/H,F)$ is locally isometric to
the Riemannian symmetric sphere $\mathrm{SO}(2n)/\mathrm{SO}(2n-1)$.

In \cite{XD-normal-homogeneous} and \cite{XD-towards-odd}, there are detailed discussion for
the uniqueness of $\mathfrak{h}=\mathrm{G}_2$ for No. 1 in Table \ref{table-4}.

To summarize, any compact connected homogeneous Finsler space in Case III-D is locally isometric to a Riemannian
symmetric sphere. By Proposition \ref{prop-classification-local}, it satisfies Assumption (I) if and only if $M$ is a Riemannian symmetric $\mathrm{S}^{2n-1}$ or a Riemannian symmetric $\mathbb{R}\mathrm{P}^{2n-1}$.

{\bf Case III-E.} Each row of Table \ref{table-5} provides
a subcase for which we can change $\mathfrak{h}$ by a suitable
 $\mathrm{Ad}(G)$-action, to make it regular in $\mathfrak{g}$.
Then Proposition \ref{prop-Case-I} provides the contradiction.

\begin{table}
  \centering
  \begin{tabular}{|c|c|c|c|c|c|}
     \hline
     No.  & $\mathfrak{g}$ & $\alpha$ & $\beta$ & Roots of $\mathfrak{h}$ \\ \hline
     1 & $\mathrm{B}_n$, $n>1$ & $e_1+e_2$ & $e_2-e_1$ &
     $\pm e_i$, $\forall i>1$; $\pm e_i\pm e_j, \forall 1<i<j\leq n$\\ \hline
     2 & $\mathrm{B}_n$, $n>1$ & $e_1+e_2$ & $e_2$ &
     $\pm e_i$, $\forall i>1$; $\pm e_i\pm e_j, \forall 1<i<j\leq n$\\ \hline
     3 & $\mathrm{C}_n$, $n>2$ & $2e_1$ & $e_1+e_2$ &
     $\pm (e_1+e_2)$; $ \pm e_i\pm e_j, \forall 3\leq i\leq n$; $\pm 2e_2, \forall i\geq 3$\\
     \hline
     4 & $\mathrm{C}_n$, $n>2$ & $2e_1$ & $2e_2$ &
     $\pm (e_1+e_2)$; $ \pm e_i\pm e_j, \forall 3\leq i\leq n$; $\pm 2e_2, \forall i\geq 3$\\
     \hline
     5 & $\mathrm{F}_4$ & $e_1+e_2$ & $e_2$ &
     $\pm e_i, \forall i>1$; $\pm e_i\pm e_j, \forall 1<i<j\leq 4$\\ \hline
     6 & $\mathrm{F}_4$ & $e_1+e_2$ & $e_2-e_1$  &
     $\pm e_i, \forall i>1$; $\pm e_i\pm e_j, \forall 1<i<j\leq 4$\\ \hline
   \end{tabular}
  \caption{Subcases in Case III-E}\label{table-5}
\end{table}

We take No. 1 in Table \ref{table-5} as the example. The argument for the other subcases are similar.

By Lemma \ref{key-lemma-2}, $\pm e_i\pm e_j$ when $1<i<j\leq n$ is a root of $\mathfrak{h}$.
In Subcase 1, Subsection 6.4 of \cite{XD-normal-homogeneous}, or Subcase 1, Subsection 4.3 of
\cite{XD-towards-odd}, it has been shown by direct calculation that,
we can change $\mathfrak{h}$ by a suitable $\mathrm{Ad}(G')$-action, where $G'$ is the
subgroup generated by the subalgebra
$\hat{\mathfrak{m}}_0=\mathbb{R}e_1+\mathfrak{g}_{\pm e_1}$,
such that we have $\mathfrak{g}_{\pm e_2}\subset\mathfrak{h}$.
Because $[\hat{\mathfrak{m}}_0,\mathfrak{g}_{\pm(e_i-e_j)}]=0$
when $1<i<j\leq n$,
we still have $\mathfrak{g}_{\pm(e_i-e_j)}\subset\mathfrak{h}$
for all $1<i<j\leq n$.

Take any nonzero $u\in\mathfrak{g}_{\pm (e_2-e_i)}$ for $i>2$, $\mathrm{ad}(u)=[u,\cdot]$ defines
a linear isomorphism
$$\mathrm{ad}(u):\hat{\mathfrak{g}}_{\pm e_2}=\mathfrak{g}_{\pm(e_1+e_2)}+\mathfrak{g}_{\pm e_2}+
\mathfrak{g}_{\pm(e_2-e_1)}
\rightarrow\mathfrak{g}_{\pm(e_1+e_i)}+\mathfrak{g}_{\pm e_i}+\mathfrak{g}_{\pm(e_i-e_1)}=
\hat{\mathfrak{g}}_{\pm e_i},
$$
and because $u\in\mathfrak{h}$, $\mathrm{ad}(u)$ maps $\mathfrak{g}_{\pm e_2}=\hat{\mathfrak{g}}_{\pm e_2}\cap\mathfrak{h}$ onto $\mathfrak{g}_{\pm e_i}\subset\hat{\mathfrak{g}}_{\pm e_i}\cap\mathfrak{h}$.
So after the $\mathrm{Ad}(G')$-action, $\pm e_i$'s for all $i>1$ are roots of $\mathfrak{h}$,
and $\mathfrak{g}_{\pm e_i}=\mathfrak{h}_{\pm e_i}$.

Until now, we have shown that all root planes of $\mathfrak{h}$ are root planes of $\mathfrak{g}$,
i.e. $\mathfrak{h}$ becomes regular in $\mathfrak{g}$ after a suitable $\mathrm{Ad}(G')$-action. Then
we can get the contradiction from Proposition \ref{prop-Case-I}.

To summarize, any odd dimensional compact connected homogeneous Finsler space $(M,F)$ in Case III-E  can not satisfy Assumption (I).

All above case by case discussions can be summarized as the following proposition.

\begin{proposition}\label{prop-Case-III}
If $(M,F)=(G/H,F)$ with $G=I_0(M,F)$ is an odd dimensional
compact connected
reversible homogeneous Finsler space in Case III
satisfying Assumption (I), then $(M,F)$ is a Riemannian symmetric
$\mathrm{S}^n$ or a Riemannian symmetric $\mathbb{R}\mathrm{P}^n$.
\end{proposition}

\begin{proof}
Table \ref{table-6} lists all the possible subcases
of compact simple $\mathfrak{g}$, the roots $\alpha$ and $\beta$
in Case III (up to Weyl group action and outer automorphisms of $\mathrm{D}_4$ and $\mathrm{E}_6$) with $|\alpha|_{\mathrm{bi}}\geq|\beta|_{\mathrm{bi}}$ and  $\theta_{\alpha,\beta}\notin\{\pi/3,2\pi/3\}$, and where they are discussed. When $\theta_{\alpha,\beta}=\pi/3$ or $2\pi/3$, the subcase belongs to Case III-A.
\end{proof}
\begin{table}
  \centering
  \begin{tabular}{|c|c|c|c|c|c|}
     \hline
     No. & $\mathfrak{g}$ & $\alpha$ & $\beta$ & $\theta_{\alpha,\beta}$ & Subcase in \\ \hline
     1 & $\mathrm{A}_n$, $n>3$ & $e_1-e_4$ & $e_3-e_2$ & $\pi/2$ &   III-B \\ \hline
     2 & $\mathrm{B}_n$, $n>1$ & $e_1+e_2$ & $e_2$ & $\pi/4$ &   III-E \\ \hline
     3 & $\mathrm{B}_n$, $n>1$ & $e_1+e_2$ & $e_2-e_1$ & $\pi/2$ &   III-E \\ \hline
     4 & $\mathrm{B}_n$, $n>3$ & $e_1+e_2$ & $-e_3-e_4$ & $\pi/2$ &   III-B \\ \hline
     5 & $\mathrm{B}_3$ & $e_1+e_2$ & $-e_3$ & $\pi/2$ &   III-D \\ \hline
     6 & $\mathrm{B}_n$, $n>3$ & $e_1+e_2$ & $-e_3$ & $\pi/2$ &   III-B\\ \hline
     7 & $\mathrm{B}_n$, $n>1$ & $e_1$ & $e_2$ & $\pi/2$ &   III-C \\ \hline
     8 & $\mathrm{B}_n$, $n>1$ & $e_1+e_2$ & $-e_1$ & $3\pi/4$ &   III-B \\ \hline
     9 & $\mathrm{C}_n$, $n>2$ & $2e_1$ & $e_1+e_2$ & $\pi/4$ & III-E \\ \hline
     10 & $\mathrm{C}_n$, $n>2$ & $2e_1$ & $2e_2$ & $\pi/2$ & III-E \\ \hline
     11 & $\mathrm{C}_n$, $n>2$ & $2e_1$ & $-e_2-e_3$ & $\pi/2$ & III-B \\ \hline
     12 & $\mathrm{C}_n$, $n>2$ & $e_1+e_2$ & $e_2-e_1$ & $\pi/2$ & III-C \\ \hline
     13 & $\mathrm{C}_n$, $n>3$ & $e_1+e_2$ & $-e_3-e_4$ & $\pi/2$ & III-B \\ \hline
     14 & $\mathrm{C}_n$, $n>2$ & $2e_1$ & $-e_1-e_2$ & $3\pi/4$ & III-B \\ \hline
     15 & $\mathrm{D}_n$, $n>2$ & $e_1+e_2$ & $e_2-e_1$ & $\pi/2$ & III-D \\ \hline
     16 & $\mathrm{D}_n$, $n>4$ & $e_1+e_2$ & $-e_3-e_4$ & $\pi/2$ & III-B \\ \hline
     17 & $\mathrm{E}_6$ & $e_1+e_2$ & $e_2-e_1$ & $\pi/2$ & III-B \\ \hline
     18 & $\mathrm{E}_7$ & $e_1+e_2$ & $e_2-e_1$ & $\pi/2$ & III-B \\ \hline
     19 & $\mathrm{E}_8$ & $e_1+e_2$ & $e_2-e_1$ & $\pi/2$ & III-B \\ \hline
     20 & $\mathrm{E}_8$ & $e_1+e_2$ & $-e_3-e_4$ & $\pi/2$ & III-B \\ \hline
     21 & $\mathrm{F}_4$ & $e_1+e_2$ & $e_2$ & $\pi/4$ & III-E \\ \hline
     22 & $\mathrm{F}_4$ & $e_1+e_2$ & $e_2-e_1$ & $\pi/2$ & III-E \\ \hline
     23 & $\mathrm{F}_4$ & $e_1+e_2$ & $-e_3$ & $\pi/2$ & III-B \\ \hline
     24 & $\mathrm{F}_4$ & $e_1$ & $e_2$ & $\pi/2$ & III-C \\ \hline
     25 & $\mathrm{F}_4$ & $e_1+e_2$ & $-e_2$ & $3\pi/4$ & III-B \\ \hline
     26 & $\mathrm{G}_2$ & $\frac32e_1+\frac{\sqrt{3}}2e_2$
     & $e_1$ & $\pi/6$ & III-A \\ \hline
     27 & $\mathrm{G}_2$ & $\frac32e_1+\frac{\sqrt{3}}2e_2$
     & $-\frac12e_1+\frac{\sqrt{3}}2e_2$ & $\pi/2$ & III-A \\ \hline
     28 & $\mathrm{G}_2$ & $-\frac32e_1+\frac{\sqrt{3}}2e_2$ & $e_1$ &
      $5\pi/6$ & III-B \\
     \hline
   \end{tabular}
  \caption{Subcases in Case III with $\theta_{\alpha,\beta}\neq\pi/3$ or $2\pi/3$}\label{table-6}
\end{table}

Theorem \ref{classification-thm-2} is just a summation of
Proposition \ref{prop-Case-I},
Proposition \ref{prop-Case-II} and Proposition \ref{prop-Case-III}.

At the end, we prove the following proposition as an application.

\begin{proposition}\label{prop-final}
Any homogeneous Finsler sphere $(M,F)$ has only one orbit of prime
closed geodesics only when it is a Riemannian symmetric sphere.
\end{proposition}

\begin{proof} By the classification of homogeneous spheres \cite{Bo1940}\cite{MS-1943},
a homogeneous Finsler sphere $(M,F)=(G/H,F)$ with $\dim M>1$, $G=I_0(M,F)$ must be one of the following:
\begin{description}
\item{\rm (1)} $G=\mathrm{SO}(n)$ when $M=\mathrm{S}^{n-1}=\mathrm{SO}(n)/\mathrm{SO}(n-1)$ and $n>2$;
\item{\rm (2)} $G=\mathrm{U}(n)$ when $M=\mathrm{S}^{2n-1}=\mathrm{U}(n)/\mathrm{U}(n-1)$ and $n>1$;
\item{\rm (3)} $G=\mathrm{Sp}(n)$ when $M=\mathrm{S}^{4n-1}=\mathrm{Sp}(n)/\mathrm{Sp}(n-1)$ and $n>0$;
\item{\rm (4)} $G=\mathrm{Sp}(n)\mathrm{U}(1)$ when $M=\mathrm{S}^{4n-1}=\mathrm{Sp}(n)\mathrm{U}(1)/\mathrm{Sp}(n-1)\mathrm{U}(1)$ and $n>1$;
\item{\rm (5)} $G=\mathrm{Sp}(n)\mathrm{Sp}(1)$ when $M=\mathrm{S}^{4n-1}=\mathrm{Sp}(n)\mathrm{Sp}(1)/\mathrm{Sp}(n-1)\mathrm{Sp}(1)$ and $n>1$;
\item{\rm (6)} $G=\mathrm{Spin}(9)$ when $M=\mathrm{S}^{15}=\mathrm{Spin}(9)/\mathrm{Spin}(7)$.
\end{description}

Now we assume $(M,F)$ satisfies Assumption (I).
By Lemma \ref{lemma-no-abel-factor}, (2) and (4) are impossible.
The homogeneous spheres in (5) and (6) only admit reversible
$G$-invariant Finsler metrics. The discussion in Case III-B above proves (6) is impossible. By Proposition \ref{prop-Case-II}, (5) is impossible.
The homogeneous spheres in (3) belong to Case I,
i.e. $H$ is regular in $G$ for this case. It is impossible by Proposition \ref{prop-Case-I}.

The only case left is (1), in which $\mathrm{S}^{n-1}=\mathrm{SO}(n)/\mathrm{SO}(n-1)$ with
$n>2$ only admits Riemannian symmetric $\mathrm{SO}(n)$-invariant metrics.

This ends the proof of the proposition.
\end{proof}

\end{document}